\documentclass[11pt]{article}
\everymath{\displaystyle}
\usepackage{caption}
\usepackage{amsfonts}
\usepackage{amssymb}
\usepackage{amsmath}
\usepackage{amsthm}
\usepackage{algorithmicx}
\usepackage{algpseudocode}
\usepackage{algorithm}

\usepackage[top=3.5cm, bottom=3.25cm, left=3cm, right=3cm, headsep=1cm]{geometry}

\usepackage{bbm}
\usepackage{bm}
\usepackage{tikz}
\usetikzlibrary{arrows}
\usetikzlibrary{decorations.markings}

\usepackage{pgfplots}

\usepackage{hyperref}

\usepackage{setspace}

 \numberwithin{equation}{section}
 
\newtheorem{thm}{thm}[section]

\newtheorem{theorem}[thm]{Theorem}
\newtheorem{lemma}[thm]{Lemma}
\newtheorem{remark}[thm]{Remark}
\newtheorem{proposition}[thm]{Proposition}
\newtheorem{definition}[thm]{Definition}
\newtheorem{corollary}[thm]{Corollary}

\newtheorem{assumption}[thm]{Assumption}

\newcommand{\norm}[1]{\left\Vert#1\right\Vert} 
\newcommand{\abs}[1]{\left|\,#1\,\right|}

\usepackage{color}

\newcommand{\QQ}{\mathbb{Q}}

\newcommand{\NN}{\mathbb{N}}
\newcommand{\EE}{\mathbb{E}}
\newcommand{\BB}{\mathbb{B}}
\newcommand{\PP}{\mathbb{P}}
\newcommand{\DD}{\mathbb{D}}

\newcommand{\E}{\mathcal{E}}

\newcommand{\RR}{\mathbb{R}}
\newcommand{\A}{\mathcal{A}}

\newcommand{\K}{\mathcal{K}}

\newcommand{\W}{\mathcal{W}}
\newcommand{\C}{\mathcal{C}}
\newcommand{\calD}{\mathcal{D}}

\newcommand{\Vs}{V_0, \ldots, V_d}

\newtheorem*{remark*}{Remark}

\usepackage{xy}
\usepackage{minitoc}
\usepackage{float}
\usepackage{graphicx}
 \usepackage{epstopdf}
\usepackage{setspace}

\makeatletter
\newtheorem*{rep@theorem}{\rep@title}
\newcommand{\newreptheorem}[2]{%
\newenvironment{rep#1}[1]{%
 \def\rep@title{#2 \ref{##1}}%
 \begin{rep@theorem}}%
 {\end{rep@theorem}}}
\makeatother

\newreptheorem{theorem}{Theorem}
\newreptheorem{lemma}{Lemma}

\title{Cubature on Wiener Space for McKean-Vlasov SDEs with Smooth Scalar Interaction}
\author{Dan Crisan, Eamon McMurray}
\begin{document}
\maketitle

\abstract{We present two cubature on Wiener space algorithms for the numerical solution of McKean-Vlasov SDEs with smooth scalar interaction. First, we consider a method introduced in \cite{nicecubature} under a uniformly elliptic assumption and extend the analysis to a uniform strong H\"ormander assumption. Then, we introduce a new method based on Lagrange polynomial interpolation.}

\section{Introduction}
In this paper, we analyse the error in two different algorithms using Cubature on Wiener space to weakly approximate the solution of a McKean-Vlasov SDE with smooth scalar interaction.
By scalar interaction, we mean that the dependence on the measure is through the integral against a scalar function, so the McKean-Vlasov SDE takes the form
\begin{equation}
\label{eq:MKVscalar}
X_t^x = x + \int_0^t V_0(X_s^x, \EE \varphi_0(X_s^x)) \, ds + \sum_{i=1}^d \int_0^t V_i(X_s^x, \EE \varphi_i(X_s^x)) \circ dB^i_s,
\end{equation}
 where $\varphi_i \in \C_b^{\infty}(\RR^N;\RR)$, $V_i \in \C_b^{\infty}(\RR^{N+1};\RR^N)$ and $B=\left(B^1, \ldots, B^d\right)$ is a Brownian motion. We wish to approximate $\EE \left[f(X_T^x)\right]$ for  $f$ Lipschitz continuous and $T>0$ a fixed time.

One common way of approaching this problem is to consider a discretisation of the equation, such as the Euler-Maruyama scheme, along with a Monte Carlo approximation. At each time step, an approximation of the law of $X_t^x$ is then given by the empirical distribution of the entire Monte Carlo population. However, estimating the error due to approximating the expectation inside the coefficients by the Monte Carlo estimator is exactly the problem one wishes to solve in the first place. This leads to a more difficult analysis than for classical SDEs. Nonetheless, this analysis has been carried out under a number of different assumptions when the coefficients have the form 
$
V_i(X_t^x, \textstyle \int h_i(X_t^x-y) \,  P_{X_t^x}(dy)) ,
$
(slightly different to \eqref{eq:MKVscalar}). This type of scheme was studied in papers by Bossy, alone \cite{bossy2005some} and along with Talay \cite{bossy97astochastic}; Kohatsu-Higa \& Ogawa \cite{kohatsuogawa}, and Antonelli \& Kohatsu-Higa \cite{antonellikohatsu}. In all of these papers the total error is composed of a discretisation error of order $\sqrt{h}$ or $h$, where $h$ is size of the largest time step, and statistical error of order $N_{MC}^{-1/2}$ where $N_{MC}$ is the number of Monte Carlo samples. A Milstein discretisation is also analysed in Ogawa \cite{ogawasome}. In \cite{tachet} Tachet des Combes proposes a deterministic numerical scheme based on discretising the PDE satisfied by the density function of the solution to \eqref{eq:MKVscalar}. More recently, in \cite{MKVMLMLC}, a Multi-Level Monte Carlo scheme has been analysed for equations of the type \eqref{eq:MKVscalar}.

Cubature on Wiener space is a high-order alternative to Monte Carlo methods. It is part of a class of methods called Kusuoka-Lyons-Victoir methods that have been shown to be highly effective in practice, see e.g.
\cite{gyurkolyons}, \cite{ninomyiavictoir}. Applications include the non-linear filtering problem \cite{crisanghazali,littlyons,crisan2013kusuoka,lee2015adaptive}, backward stochastic
differential equations \cite{crisanmansolving,crisanmansecond} and  calculating Greeks \cite{teichmann2006calculating} in finance. Convergence of the cubature approximation for some path dependent functionals has also been shown in \cite{frizbayer}. The starting point of the Cubature on Wiener space method is to view $X^{x}_T$ as a functional of the Brownian path $(B_{s})_{s \in [0,T]}$, say
\[
X^{x}_T = \alpha_{x,T}((B_{s})_{s \in [0,T]}),\:\:\:\: \alpha_{x,T} :\C_0([0,T];\RR^d) \to \RR^N,
\]
and to view the expectation $\EE[f(X_T^{x})]$ as an integral over the Wiener space
\begin{equation}
\label{eq:integral}
\EE[f(X_T^{x})] = \int_{\Omega} \left( f \circ \alpha_{x,T} \right) (\omega) \, \PP(d \omega),
\end{equation}
where $\Omega=\C_0([0,T];\RR^d)$ and $\PP$ is the Wiener measure.
The key idea of the cubature on Wiener space method is that one can approximate such integrals by replacing the Wiener measure, $\PP$, by a discrete measure supported on finitely many bounded variation paths, called a cubature measure. If the cubature measure is chosen so that iterated Stratonovich integrals of Brownian motion up to some order have the same expectation under the cubature and Wiener measures and the time interval $[0,T]$ is small, then by considering the Stratonovich-Taylor expansion of the solution of the SDE, one can show that the target expectations under the cubature and Wiener measures agree up to some high order error. When the time interval is not small, $[0,T]$ can be partitioned into sub-intervals and the approximation performed over each sub-interval.
In the original paper \cite{lyons2004cubature}, dealing with ordinary SDEs, evaluating the functional $\alpha_{x,T}$ at a bounded variation path $\omega$ amounts to solving an rdinary differential equation (ODE) and evaluating integrals such as \eqref{eq:integral} under a cubature measure amounts to computing weighted sums of solutions of ODEs. The complication for McKean-Vlasov equations like \eqref{eq:MKVscalar} is that the functional $\alpha_{x,T}$ depends on the paths $\left\{ (\EE \varphi_i(X^x_s))_{s \in [0,T]}:i=0, \ldots, d \right \}$ which are unknown. Instead, one must include an approximation of the functional $\alpha_{x,T}$ in the design of the algorithm.

To our knowledge, the first algorithm involving Cubature on Wiener Space in relation to McKean-Vlasov SDE was introduced by Chaudru de Raynal \& Garcia-Trillos \cite{nicecubature}.
Their idea is to partition $[0,T]$ into $\{0=t_0 < t_1 < \ldots < t_n=T \}$ and over the interval $[t_j, t_{j+1}]$ to replace $ \EE\varphi_i(X^x_t)$ appearing in the coefficients
with the cubature approximation of the Taylor exapnsion of the path $t \mapsto \EE\varphi_i(X^x_t)$ around $t_j$ up to some order, $q$.
The global error can, as in the original case, be decomposed as a sum of local errors, and these local errors naturally split into an error due to the approximation of $\EE\varphi_i(X^x_t)$ in the coefficients, and an error due to replacing the Wiener measure by a cubature measure.
The authors consider the case of smooth and  bounded uniformly elliptic coefficients and prove that the error is of order $n^{-[(q+1) \wedge (l-1)/2]}$ where $n$ is the number of time steps and $l$ is the degree of the cubature formula. This is the first algorithm we consider. We show how to extend the error analysis to the case when the coefficients satisfy a uniform strong H\"ormander condition. One of the reasons the authors of \cite{nicecubature} choose to impose a uniformly elliptic condition on the coefficients of equation \eqref{eq:MKVscalar} is the lack of available sharp derivative estimates for  time-inhomogeneous parabolic PDEs (which are necessary for the error analysis) under any more general conditions. For this reason, a secondary goal of this work is to develop derivative estimates for  time-inhomogeneous parabolic PDEs under more general conditions and to analyse the error for the cubature on Wiener space algorithm in this case.

In the second algorithm, which we call the Lagrange interpolation method, over the interval $[t_j, t_{j+1}]$, one simply replaces $\EE\varphi_i(X^x_t)$ with the Lagrange polynomial which interpolates the cubature approximation of $\EE \varphi_i(X_{t}^x)$ at the previous $r$ points in the time partition.

\subsection{Cubature on Wiener Space}

First, we detail exactly what we mean by a cubature formula on Wiener space. We need to introduce notation for iterated integrals with respect to components of the $(d+1)$-dimensional process $(B^0, B^1, \ldots, B^d)$ consisting of time and the $d$-dimensional Brownian motion. We use the following notation for multi-indices on $\{0, \ldots, d\}$:
\[
\A := \{\emptyset \} \cup \cup_{k \geq 1 } \{0, 1, \ldots, d \}^k \quad \text{ and } \quad \A_{1} := \A \setminus \{ \emptyset, (0) \} .
\] 
We endow $\A$ with the concatenation operation
\[
\alpha \ast \beta := (\alpha_1, \ldots,\alpha_k, \beta_1, \ldots,
\beta_l), \ \ \text{where } \alpha = (\alpha_1,\ldots,\alpha_k), \ \ \ \beta =
(\beta_1,\ldots,\beta_l) \in \mathcal{A}.
\]
and we define $\alpha'= (\alpha_1)$ and $-\alpha:=(\alpha_2, \ldots, \alpha_k)$,  so that $\alpha = \alpha' \ast -\alpha$.
We define the following $n$-tuples lengths:
\begin{equation*}
\abs{\alpha} :=
\begin{cases}
k, & \text{if   } \alpha = (\alpha_1, \ldots, \alpha_k), \\
0, & \text{if   } \alpha = \emptyset.
\end{cases} \: \: \:\:\:
\norm{\alpha} := \abs{\alpha} + \text{card }\{i: \alpha_i=0, i =
1, \ldots, d\},
\end{equation*}
and define the set $\mathcal{A}(l) := \{ \alpha \in \mathcal{A} : \| \alpha \| \leq l \} $ and define $\A_1(l)$ similarly.
For $\alpha \in \A$, we denote by $I^{\alpha}_{t,s}(Y)$ the iterated Stratonovich integral of the process $Y$ over the interval $[t,s]$:
\begin{equation*}
I^{\alpha}_{t,s}(Y):=\int_{t}^{s} \int_{t}^{s_{n}} \cdots
\int_{t}^{s_{2}} Y_{s_1} \circ dB ^{\alpha_{1}}_{s_{1}} \cdots \circ d B
_{s_{n-1}}^{\alpha_{n-1}} \circ d B _{s_n}^{\alpha_{n}}.
\end{equation*}%
Similarly, for a bounded variation path $\omega= (\omega^1, \ldots, \omega^d) \in \C_{bv}([t,s];\RR^d)$ we set $\omega^0(s)=s$ and denote the iterated integral of a process $Y$ by $I^{\alpha}_{t,s}[\omega](Y)$:
\begin{equation*}
I^{\alpha}_{t,s}[\omega](Y):=\int_{t}^{s} \int_{t}^{s_{n}} \cdots
\int_{t}^{s_{2}} Y_{s_1}  d\omega ^{\alpha_{1}}_{s_{1}} \cdots  d \omega
_{s_{n-1}}^{\alpha_{n-1}} \,d \omega _{s_{n}}^{\alpha_{n}}.
\end{equation*}%
With this notation in hand, we can define a cubature formula.

\begin{definition}[Cubature formula \cite{lyons2004cubature}]
	A set of $N_{Cub}$ bounded variation paths, \newline
	$
	\omega _{1},\ldots ,\omega _{N_{Cub}} \in \mathcal{C}_{bv}([0,1];\mathbb{R}^{d})$, for
	some $N_{Cub}\in \mathbb{N}$, together with some weights \newline $\lambda _{1},\ldots
	,\lambda _{N_{Cub}}\in \mathbb{R}^{+}$\ such that $\textstyle \sum_{j=1}^{N_{Cub}} \lambda_{j}=1$  define a cubature
	formula on Wiener Space of degree $l$ if, for any $\alpha \in \A(l)$, 
	\begin{equation*}
	\EE[I^{\alpha }_{0,1}(1)] = \sum_{j=1}^{N_{Cub}}
	\lambda_{j} \, I^{\alpha }_{0,1}[\omega _{j}](1).
	\end{equation*}%
\end{definition}
\noindent We note that for a given $l \in \mathbb{N}$, Lyons \& Victoir \cite{lyons2004cubature} proved that there exists a cubature formula on Wiener Space of degree $l$, with concrete examples given, for certain pairs $(l,d)$, in \cite{lyons2004cubature} and \cite{gyurkolyons}. From the scaling properties of the Brownian motion we can deduce, for $0 \leq t <s$, 
\begin{equation*}
\EE[I^{\alpha }_{t,s}(1)] = \sum_{j=1}^l
\lambda_{j} \, I^{\alpha }_{t,s}[\omega_j(t,s)](1),
\end{equation*}%
where $\omega_j(t,s)$ is the  re-scaled path defined by $\textstyle \omega_{j}(t,s) \left( u \right) =\sqrt{s-t}\, \omega
_{j}\left( \frac{u}{s-t}\right)$, $u\in \left[ t,s\right] $. In other words,
the expectation of the iterated Stratonovich integrals $I^{\alpha }_{t,s}(1)$ with $\alpha \in \A(l)$ is
the same under the Wiener measure as it is under the cubature measure, 
\begin{equation*}
\mathbb{Q}_{t,s}:=\sum_{j=1}^{N_{Cub}} \lambda_{j} \, \delta_{\omega_j(t,s)}.
\end{equation*}

Once we have a cubature measure $\QQ$ and a partition $\Pi_n$, we can extend this to a measure $\QQ^{\Pi_n}$ on $[0,T]$, supported on $(N_{cub})^n$ paths along a tree. We use the notation $\mathcal{M}_k$ to denote multi-indices over $\{1, \ldots, N_{Cub}\}$ of length exactly $k$. We use this set to index the nodes in the cubature tree after $k$ time-steps or, equivalently, the unique path leading to that node. To create the tree, one first creates the paths by concatenating the re-scaled paths: for $p=(p_1, \ldots, p_n) \in \mathcal{M}_n$, define the path
\[
\omega_p(t) = \omega_p(t_{i-1}) + \omega_{p_i}(t_{i-1}, t_i)(t) \quad \text{ when } t \in [t_{i-1}, t_i).
\]
Then, one can attach a new weight to each path by
\[
\Lambda_{p} := \prod_{p_i \in p} \lambda_{p_i}.
\]
Finally, we can define a measure on all paths along the tree by
\[
\QQ^{\Pi_n} := \sum_{p \in \mathcal{M}_n} \Lambda_{p} \, \delta_{ \omega_{p} }.
\]

\subsection{Outline \& Main Results}

In this section, let us make more precise our contribution. We introduce the following processes
\begin{align}
	\label{eq:MKVscalarlate}
	X_t^{s,x} = x + \sum_{i=0}^d \int_s^t V_i(X_u^{s,x}, \EE  \varphi_i (X_u^{s,x})  ) \circ dB^i_u,
\end{align}
and
\begin{align}
	\label{eq:MKVscalardecoupled}
	X_t^{s,x,y} = x + \sum_{i=0}^d  \int_s^t V_i(X_u^{s,x,y}, \EE  \varphi_i (X_u^{0,y})  ) \circ dB^i_u.
\end{align}
The first is just the McKean-Vlasov SDE started from $x$ at time $s$. The second process is also started from $x$ at time $s$ but with the path $u \mapsto \EE [ \varphi_i (X_u^{0,y})] $ appearing in the coefficients instead of the McKean-Vlasov term. This process is therefore not a true McKean-Vlasov process but an SDE with coefficients depending on time and a parameter, $y$. We introduce the operators
\begin{align*}
	& P_{s,t} f(x) := \EE \left[ f(X_t^{s,x}) \right]\quad \text{ and } \quad  P_{s,t}^y f(x):= \EE \left[ f(X_t^{s,x,y}) \right].
\end{align*}
We note that
$
X_T^{0,x} = \left. X_T^{0,x,y} \right|_{y=x},
$
so the quantity we wish to compute is
\[
P_{0,T} f(x) = \left. P^y_{0,T}f(x) \right|_{y=x}.
\]
Now, let us denote by $E_t^{y}(\varphi_i)$ a generic approximation of $ \EE [ \varphi_i (X_t^{0,y}) ]$. Later, we will introduce specific approximations $E_t^{T,y}(\varphi_i)$ and $E_t^{L,y}(\varphi_i)$, corresponding to the Taylor and Lagrange interpolation methods respectively. We then introduce the approximating process
\begin{align*}
	^EX_t^{s,x,y} = x + \sum_{i=0}^d  \int_s^t V_i(^E X_u^{s,x,y}, E_u^{y}(\varphi_i)) \circ dB^i_u,
\end{align*}
and the operators
\begin{align*}
	P_{s,t}^{E,y} g(x): =  \EE\left[ g (^EX_t^{s,x,y})\right] \quad \text{ and } \quad 	 Q_{s,t}^{E,y} g(x): =  \EE_{\QQ_{s,t}}\left[ g (^EX_t^{s,x,y})\right].
\end{align*}
In a similar way, we will denote
the local approximation operator by $Q_{s,t}^{E,x}$ and, once a partition $\Pi_n$ of $[0
,T]$ is fixed, we define
\[
Q^{E,x,\Pi_n}_{0,t}: = Q_{0,t_1}^{E,x} \cdots Q^{E,x}_{t_j,t} \quad \text{ for } t \in [t_j,t_{j+1}).
\]
Then, $Q^{E,x,\Pi_n}_{0,T}f(x)$ will be the final approximation of $P_{0,T}f(x)$, with the global error
\[
\E(T,x,l,\Pi_n)  := \left( P_{0,T} - Q^{E,x,\Pi_n}_{0,T} \right) f(x) .
\]	
We note that 	
\[
\sup_{x \in \RR^N} \left| \E(T,x,l,\Pi_n) \right| \leq \sup_{x,y \in \RR^N} \left| \left( P^y_{0,T} - Q^{E,y,\Pi_n}_{0,T} \right) f(x) \right|.
\]
Now, for fixed $y$, $\{P^y_{s,t}: 0 \leq s \leq t  \leq T\}$ forms a two-parameter semigroup of operators. This allows us to decompose the global error the scheme as follows
\begin{align*}
	\left[ P^y_{0,T} - Q^{E,y}_{0,t_1} Q^{E,y}_{t_1,t_2} \cdots Q^{E,y}_{t_{n-1}, t_n} \right] f(x) &= \sum_{j=0}^{N-1} Q^{E,y}_{0,t_j} \left[ P^y_{t_j,t_{j+1}}- Q_{{t_j,t_{j+1}}}^{E,y} \right]P^y_{t_{j+1},T} f(x) .
\end{align*}
Then, since $ \left\| Q^{E,y}_{0,t_j} \phi \right\|_{\infty} \leq \|\phi\|_{\infty}$, we are left to estimate the \emph{local error} 
\[
\left[ P^y_{t_j,t_{j+1}}- Q_{{t_j,t_{j+1}}}^{E,y} \right] u^y(t_{j+1},x)
\]
uniformly in $x$ and $y$, 
where $u^y(t,x) := P^y_{t,T} f(x)$ solves a parabolic PDE with coefficients depending on the parameter $y \in \RR^N$. The resulting error analysis relies on regularity estimates for the solution of this PDE.

Now, let us specify what the approximation $E^x_t(\varphi_i)$ is for each scheme. First, the Taylor method:
we wish to perform a Taylor expansion of the path $t \mapsto  \EE \varphi_i(X^{0,x}_t)$, but since the coefficients in the SDE satisfied by $X^{0,x}$ are of the form $V_i \left(X^{0,x}_t, \EE \varphi_i(X^{0,x}_t) \right)$, we instead consider the Taylor expansion of this more general form.
For a pair of functions 
$g \in \C^{\infty}_b(\RR^N \times \RR;\RR)$ and $\varphi \in \C^{\infty}_b(\RR^N;\RR)$, It\^{o}'s formula yields
\begin{align}
	\label{eq:Taylorexp}
	\begin{split}
		\EE	&\left[g \left(X^{0,x}_t, \EE \varphi(X^{0,x}_t) \right)  \right] = g \left(x, \varphi(x) \right)  \\
		& \quad + \int_0^t  \EE \left[ \left(\mathcal{L}^x_u g \right)\left(X^{0,x}_u, \EE \varphi(X^{0,x}_u) \right) \right]  + \EE \left[ (\partial_y g) \left(X^{0,x}_u, \EE \varphi(X^{0,x}_u) \right) \right] \EE \left[ \left(\mathcal{L}^x_u \varphi\right)(X^{0,x}_u)  \right] \, du,
	\end{split}
\end{align}
where $\partial_y$ is the derivative in the second argument of $g$ and $\mathcal{L}^x_s$ is the differential operator
\begin{align*}
	& \mathcal{L}_s^x: = V_0(\cdot, \EE \varphi_0(X_s^{0,x}))  + \frac{1}{2} \sum_{i=1}^d V_i(\cdot, \EE \varphi_i(X_s^{0,x}))^2 .
\end{align*}
Note that each term under the integral in the right hand side of \eqref{eq:Taylorexp} is again a product of terms of the form $\EE	\left[g \left(X^{0,x}_u, \EE \varphi(X^{0,x}_u) \right)  \right]$ for different functions $g$ and $\varphi$.
Let us denote by $\mathcal{S}$ the set of all paths $(S_t)_{t \in [0,T]}$ of the form
\[
S_t= \EE \left[ g \left(X^{0,x}_t, \EE \varphi(X^{0,x}_t) \right)  \right]
\]
for some
$g \in \C^{\infty}_b(\RR^N \times \RR;\RR)$, and $\varphi \in \C^{\infty}_b(\RR^N;\RR)$
and denote by $\bar{\mathcal{S}}$ the set of all paths $(\bar{S}_t)_{t \in [0,T]}$ of the form
\[
\bar{S}_t=	\sum_{i=1}^{n_1} \prod_{j=1}^{n_2} S^{i,j}_t
\]
where $n_1$ and $n_2$ are positive integers and $S^{i,j} \in \mathcal{S}$ for $i=1, \ldots, n_1$ and $j=1, \ldots, n_2$.
For $t \in [0,T]$, we introduce the notation $\mathcal{T} : \mathcal{S} \to \bar{\mathcal{S}}$ for the map
\begin{align*}
	\EE & \, g  \left(X^{0,x}, \EE \varphi(X^{0,x}) \right)  \\
	&  \mapsto \EE \left[ \left(\mathcal{L}^x g \right)\left(X^{0,x}, \EE \varphi(X^{0,x}) \right)\right]  + \EE \left[ (\partial_y g) \left(X^{0,x}, \EE \varphi(X^{0,x}) \right) \right] \EE \left[ \left(\mathcal{L}^x \varphi\right)(X^{0,x})  \right],
\end{align*}
so that the expansion of $G_t:=\EE	\left[g \left(X^{0,x}_t, \EE \varphi(X^{0,x}_t) \right)  \right]$ in \eqref{eq:Taylorexp} can alternatively be written as
\begin{equation}
\label{eq:G}
G_t =G_0 + \int_0^t \mathcal{T}_s (G) \, ds
\end{equation}
For the Taylor expansion, we would like to apply $\mathcal{T}$ to the term $\mathcal{T} (G)$. To do so, we extend $\mathcal{T}$ to an operator from $\bar{\mathcal{S}}$ to itself by linearity
\[
\mathcal{T} \left(	\sum_{i=1}^{n_1}  S^{i}\right) = \sum_{i=1}^{n_1} \mathcal{T} \left( S^{i} \right),
\]
and a product rule
\[
\mathcal{T} \left(	 \prod_{j=1}^{n_2} S^{j}\right) =  \sum_{k=1}^{n_2} \prod_{j=1,j \neq k}^{n_2} S^{j} \: \mathcal{T} \left( S^{k} \right).
\]
Since now we have $\mathcal{T}$ defined as an operator on $ \bar{\mathcal{S}}$, we can iterate the expansion in \eqref{eq:G} to get the Taylor expansion of order $q \geq 1$
\begin{equation}
\label{eq:TaylorG}
G_t = \sum_{k=0}^q \frac{t^k}{k!} \,(	\mathcal{T}_{0})^k (G)  + \frac{1}{(q+1)!} \int_0^t  \,(	\mathcal{T}_{s})^{q+1} (G) \, ds .
\end{equation}
Now, for all $s \in [0,T]$, $k \geq 1$ and $G \in \bar{\mathcal{S}}$ we define $ \left(\mathcal{T}_s^{Q}\right)^k(G)$ to be the same expression as $ \left(\mathcal{T}_s\right)^k(G)$ with all expectations under $\PP$ replaced by expectations under $\QQ^{\Pi_n}$.
Then, the approximation of $\EE \varphi_i(X^{0,y}_s)$  for the Taylor method, which we henceforth denote by $E^{T,y}_{s}(\varphi_i)$, is
\begin{align*}
	& E^{T,y}_{s}(\varphi_i) := \sum_{k=0}^q \frac{1}{k!} \,(	\mathcal{T}^{Q}_{t_j})^k (\EE \varphi_i(X^{0,y})) \, (s-t_j)^k \quad \text{ for } s \in [t_j,t_{j+1}).
\end{align*}
We make some explicit computations of this type in Example ... in Section \ref{sec:NumEx}.

To define the approximation for the Lagrange interpolation method, we denote by $L \left[ \left\{(t_i, x_i)\right\}_{i=1}^k \right]$ the Lagrange interpolating polynomial of degree at most $k-1$ with $L \left[ \left\{(t_i, x_i)\right\}_{i=1}^k \right](t_j) = x_j$ for all $j=1, \ldots, k$. We define the approximation of order $r$ by
\[
E^y_t(\varphi_i): =  L \left[ \left\{ \left(t_{j-k} ,\EE^{\QQ^{\Pi_n}} \varphi_i \left (X^{0,y}_{t_{j-k}} \right)  \right) \right\}_{k=0}^{j \wedge (r-1)} \right]  (t) \quad \quad t \in [t_j,t_{j+1}]
\]
In other words, $E^{y}_{t}(\varphi_i)$ for $t \in [t_j,t_{j+1}]$ is the unique polynomial of degree minimal degree which passes through the points $\left(t_{j+1-((j+1)\wedge r)}, \EE^{\QQ^{\Pi_n}} \varphi(X^{0,y}_{t_{j+1-((j+1)\wedge r)}}) \right), \ldots, \left(t_{j}, \EE^{\QQ^{\Pi_n}} \varphi(X^{0,y}_{t_{j}})\right) $. That is, if the time index $j$ is greater than $r-1$, we interpolate through the last $r$ cubature approximations of $\EE \varphi(X^{0,y}_{t})$ along the partition. If $j < r-1$, we interpolate through all of the available previous points. We now detail both algorithms.


\begin{algorithm}[H]
	\caption{Taylor method}\label{algo:Taylor}
	\begin{algorithmic}[1]
		\State Set $(X^{\emptyset},\Lambda_0) =(x, 1)$
		\For {$0 \leq j \leq n-1$} 
		\State Let 
		$
		\mathcal{T}^{Q}
		$
		be as in the main text.
		\For {$p \in \mathcal{M}_j$}
		\For{$0 \leq i \leq d$}
		\State Set $\displaystyle E^T_i(t)=\sum_{k=0}^{q} \frac{1}{k!}(t-t_j)^k (	\mathcal{T}^{Q}_{t_j})^k (\EE \varphi_i(X^{0,x}))$
		\EndFor
		\For {$1 \leq l \leq N_{cub}$}
		\State Define $X_{t_{j+1}}^{p * l}$ as the solution of the ODE:
		\begin{align*}
			& dX^{p * l}_t =  \sum_{i=0}^{d} V_i \left(X_t^{p * l}, E^T_i(t)\right) \, d\omega_{l}^{i}(t_j,t_{j+1})( t ),\\
			& X^{p *l}_{t_j}=X^{p}_{t_j}.
		\end{align*}
		\State Set the associated weight: $\Lambda_{p * l} = \Lambda_{p} \lambda_l$
		\EndFor
		\For{$0 \leq i \leq d$}
			\State Store $\sum_{p \in \mathcal{M}_{j+1} } \Lambda_p \varphi_i(X^p_{t_{j+1}})$.
		\EndFor
		\EndFor
		\EndFor
		\State Final approximation of $\EE\left[f(X_T^x)\right]$ is
		\[
		\sum_{p \in \mathcal{M}_n} \Lambda_{p} f(X^{p}_{t_n})
		\]
	\end{algorithmic}
\end{algorithm} 

\begin{algorithm}[H]
	\caption{Lagrange interpolation method}\label{algo:lagrange}
	\begin{algorithmic}[1]
		\State Set $(X^{\emptyset},\Lambda_0) =(x, 1)$
		\For {$0 \leq j \leq n-1$} 
		\For {$p \in \mathcal{M}_j$}
		\For{$0 \leq i \leq d$}
		\State Set $
		E^L_i(t): =  L \left[ \left\{ \left(t_{j-k} ,\sum_{p \in \mathcal{M}_{j-k} } \Lambda_p \varphi_i(X^p_{t_{j-k}}) \right) \right\}_{k=0}^{ j \wedge (r-1) } \right]  (t) $
		\EndFor
		\For {$1 \leq l \leq N_{cub}$}
		\State Define $X_{t_{j+1}}^{p * l}$ as the solution of the ODE:
		\begin{align*}
			& dX^{p * l}_t =  \sum_{i=0}^{d} V_i \left(X_t^{p * l}, E^L_i(t)\right) \, d\omega_{l}^{i}(t_j,t_{j+1})( t ),\\
			& X^{p *l}_{t_j}=X^{p}_{t_j}.
		\end{align*}
		\State Set the associated weight: $\Lambda_{p * l} = \Lambda_{p} \lambda_l$
		\EndFor
		\For{$0 \leq i \leq d$}
		\For{$0 \leq k \leq ((j+1) \wedge r)$}
		\State Store $\sum_{p \in \mathcal{M}_{j+1-k} } \Lambda_p \varphi_i(X^p_{t_{j+1-k}})$.
		\EndFor
		\EndFor
		\EndFor
		\EndFor
		\State Final approximation of $\EE\left[f(X_T^x)\right]$ is
		\[
		\sum_{p \in \mathcal{M}_n} \Lambda_{p} f(X^{p}_{t_n})
		\]
	\end{algorithmic}
\end{algorithm}

\begin{remark}
	\label{rem:algocomp}
	The Taylor method requires finding an expression for $(\mathcal{T}_{t})^k (\EE \varphi_i(X^{0,y}))$ for $k=1, \ldots, q$ and $i=0, \ldots, d$ either by hand or using some symbolic computation. The Lagrange interpolation method does not require this; the interpolating polynomial is defined at each time step as part of the algorithm.
\end{remark}
\noindent Now, we state the main assumptions. First, we introduce the notation $V_{[\alpha]}$ for iterated Lie brackets of the vector fields.  In this setting each $\Vs : \RR^N \times \RR\to \RR^N$ and we think of these as vector fields $V_i(\cdot,x')$ on $\RR^N$ parametrised by the second variable, $x' \in \RR$, with the Lie Bracket between any two given by
\[
	[V_i,V_j](x,x') = \partial_x V_j(x,x') V_i(x,x') -\partial_x V_i(x,x') V_j(x,x')
\]
where $\partial V_i(x,x'):= (\partial_{x_l} V_i^k(x,x'))_{1 \leq k,l \leq N}$ is the Jacobian matrix of $V_i$ and similarly for $\partial_x V_j$. Then, for $\alpha \in \bigcup_{k \geq1} \{1, \ldots, N\}^k$ and $i \in \{1, \ldots, N\}$, we define inductively
\begin{align*}
V_{[i]}:=V_{i}, \quad \quad V_{[\alpha*i]}:= [V_i,V_{[\alpha]}] \quad .
\end{align*}
Now we are able to state our assumptions.
\begin{assumption}
	\label{ass:MKVcub}
	\begin{itemize}
		\item[(A1):] Uniform strong H\"ormander condition: there exist $\delta>0$ and $m \in \mathbb{N}$ such that for all $\xi \in \RR^N$,
		\begin{align*}
			\inf_{(x,x') \in \RR^N \times \RR} \sum_{\alpha \in \cup_{k =1}^m \{1, \ldots, N\}^k} \langle V_{[\alpha]}(x,x'), \xi \rangle^2 \geq \delta \, \left|\xi\right|^2
		\end{align*}
		\item[(A2):] Smoothness of coefficients:
		\begin{align*}
			&\varphi_i \in \C^{\infty}_b(\RR^N;\RR), \quad V_i \in \C^{\infty}_b(\RR^N \times \RR;\RR^N ) \quad i=0, \ldots, d 
		\end{align*}
		\item[(A3):] We assume the paths in any cubature formula we use are absolutely continuous.
	\end{itemize}
\end{assumption}
As is common with cubature on Wiener space methods, when the terminal function $f$ is not smooth, we will use an uneven partition of the time interval $[0,T]$.
Here, we introduce the Kusuoka partition and a modified version. We denote by $\Pi^{\gamma}_n$ the Kusuoka \cite{Kusuoka2} partition of the interval $[0,T]$ with $(n+1)$ points and parameter $\gamma \geq 1$, defined by 
\begin{align*}
& t_j = T \left(1- \left(1-\frac{j}{n} \right)^{\gamma} \right) \quad \text{ for }  j=0, \ldots, n-1, \\
& t_n=T.
\end{align*}

We denote by $\Pi^{\gamma,r}_n$ the modified Kusuoka partition, with $r$ smaller steps at the start whose size is determined by the overall order of the method we require. It is defined as follows: for a fixed integer $r$ and real parameter $\gamma$, we fix the first $(r+1)$ points as $t_0=0$,  $t_{k+1}-t_k=T n^{-r/(k+1)}$ for $k=0, \ldots, r-1$. Thereafter and we split the rest of the interval $[t_r,T]$ using the Kusuoka partition, i.e.
\begin{align*}
	& t_j =  \left(T- t_r \right) \textstyle \left(1-\left(1-\frac{j-r}{n-r}\right)^{\gamma} \right) + t_r \quad \quad j\in \{r+1,\ldots,  n-1\} \\
	& t_n=T
\end{align*}

\noindent Then we have the following result, which is the main result of this work.

\begin{theorem}
	\label{th:MKVCub}
	Let $f \in \C^{\infty}_b(\RR^N;\RR)$. Then, assuming (A2), the error for the Taylor method satisfies the following
	\begin{equation*}
		\label{eq:errsmoothTaylor}
		\sup_{x \in \RR^N} \left| \E(T,x,l,\Pi_n) \right| \leq C \, \sum_{j=0}^{n-1} (t_{j+1}-t_j)^{A(q,l)},
	\end{equation*}
	where $A(q,l):= (q+2) \wedge (l+1)/2$. Under the same assumptions, the error in the Lagrange interpolation method is
	\begin{equation*}
		\label{eq:errsmoothcoupled}
		\sup_{x \in \RR^N} \left|\E(T,x,l,\Pi_n) \right| \leq C \, \sum_{j=0}^{n-1} \bigg\{  \frac{(t_{j+1}-t_j)}{((j+1) \wedge r)!}\prod_{k=0}^{j \wedge (r-1)} (t_{j+1}-t_{j-k}) 
		+ (t_{j+1}-t_j)^{(l+1)/2} \bigg\}.
	\end{equation*}
	Now, suppose $f$ is only Lipschitz continuous. Assuming (A1)-(A3) and that we use the Kusuoka partition $\Pi^{\gamma}_n$ with $\gamma >l-1$,  we can bound the error in the Taylor method according to the size of $m$
	\begin{align}
		\label{eq:errLipTaylorm1}
		m=1: \quad \quad &\sup_{x \in \RR^N} \left|\E(T,x,l,\Pi^{\gamma}_n) \right| \leq C \, n^{-B(q,l)-1/2}, \\
		\label{eq:errLipTaylorm2}
		m\geq 2: \quad \quad &\sup_{x \in \RR^N} \left|\E(T,x,l,\Pi^{\gamma}_n) \right| \leq C \, n^{-B(q,l)},
	\end{align}
	where $B(q,l)=(q+ \textstyle \frac{1}{2}) \wedge \frac{l-2}{2}$.
	Assuming (A1)-(A3) and that we use the modified Kusuoka partition $\Pi^{\gamma,r}_n$ with $ \gamma \in (l-1,l)$, we can bound the error in the Lagrange interpolation method according to the size of $m$
	\begin{align}
		\label{eq:errLipLagm1}
		m=1: \quad \quad & \sup_{x \in \RR^N} \left|\E(T,x,l,\Pi^{\gamma,r}_n) \right| \leq C \, n^{-D(r,l)-1/2} \,  (1-r/n)^{-l/2} ,\\
		\label{eq:errLipLagm2}
		m \geq 2 : \quad \quad &\sup_{x \in \RR^N} \left|\E(T,x,l,\Pi^{\gamma,r}_n) \right| \leq C \, n^{-D(r,l)} \,  (1-r/n)^{-l/2},
	\end{align}
	where $D(r,l)=(r- \textstyle \frac{3}{2}) \wedge \textstyle \frac{l-2}{2}$. 
\end{theorem}

\begin{remark}
	\begin{enumerate}
		\item Let us comment on the term $(1-r/n)^{-l/2}$ appearing in the error for the Lagrange interpolation method. This term comes from the need to take small steps at the start (for an accurate polynomial approximation) and end of the partition (due to blow up of the derivatives of $u^y(t,\cdot)$). We need to take $r$ small steps at the start, leaving $n-r$ steps split in the style of the Kusuoka partition. This term reflects the need to balance the size of $r$ and $n$. Since it goes to zero as $n \to \infty$, it can be bounded by a constant for sufficiently large $n$. 
		For example, in the case $m=1$ for a second order method, one must use a cubature formula of degree $l=5$ and choose $r=4$. Then, for $n\geq 10$ $(1-r/n)^{-l/2}=(1-4/n)^{-5/2}<3.6$.
		\item The uniformly elliptic case covered by \cite{nicecubature} is the case $m=1$. Choosing the parameter $r$ appropriately, we also recover the same rate for the Lagrange interpolation method up to the multiplication of the term $(1-r/n)^{-l/2}$.
		\item In the case where $m \geq 2$, we lose $1/2$ an order of convergence. This is due to the difference in the way we split the error, which we explain in Remark \ref{rem:expl}
	\end{enumerate}
\end{remark}

\section{Preliminary results}

First, we have a lemma on existence, uniqueness and moment bounds for the solution of equation \eqref{eq:MKVscalardecoupled}.

\begin{lemma}
	\label{lem:MKVbds}
	Under assumption (A2), there exist unique strong solutions to equations \eqref{eq:MKVscalarlate} and \eqref{eq:MKVscalardecoupled}. Moreover, for all $s \in [t,T]$, the mapping $x \mapsto X_s^{t,x,y}$ is $\PP$-a.s. smooth, and for all multi-indices $\eta$ on $\{1, \ldots, N\}$,
	\begin{equation}
	\label{eq:decMKVderivbd}
	\sup_{x,y \in \RR^N} \left\| \partial^{\eta}_x X^{t,x,y}_s \right\|_p < \infty \quad \forall p \geq 1.
	\end{equation}
\end{lemma}

\begin{proof}
	Under assumption (A2), existence and uniqueness of strong solutions to equations \eqref{eq:MKVscalarlate} and \eqref{eq:MKVscalardecoupled} is easy to prove and can be found in, for example, \cite{Jourdain_nonlinearsdes}. Now, we note that we can view $\left(X^{t,x,y}_s\right)_{s \in [t,T]}$ as the solution of an SDE with coefficients
	\[
	W^y_i(s,z):= V_i(z, \EE \varphi_i (X^{0,y}_s)),
	\]
	depending on time and a parameter. Due to Assumption \ref{ass:MKVcub} (A2) $(s,z) \mapsto W^y_i(s,z)$ is smooth, with bounded derivatives of all orders, with all bounds uniform in $y$. The differentiability in $z$ is assumed, and the differentiability in $s$ comes from the smoothness of each $\varphi_i$, which allows us to apply It\^o's formula to $\varphi_i (X^{0,y}_s)$, giving
	\[
	\partial_s  \left[\EE \varphi_i (X^{0,y}_s) \right] = \EE \left[ (\mathcal{L}^y_s \varphi_i) (X^{0,y}_s) \right].
	\]
	Then, Kunita \cite[Theorem 4.6.5]{Kun84} guarantees that the moment bound \eqref{eq:decMKVderivbd} holds.
\end{proof}

In the next lemma, we collect some results on the regularity of $u^y(t,x)$ and the pure cubature part of the error. We use the notation,  for $\psi \in \C^2_b(\RR^N;\RR)$,
\[
\|\psi\|_{2, \infty}:= \sup_{x \in \RR^N} \left\{  |\psi(x) | + |\nabla \psi(x)| + |\nabla^2 \psi(x)| \right\}.
\]

\begin{lemma}
	\label{lem:MKCcubreg}
	Let $j \in \{0, \ldots, n-2\}$ and $t \in [0,T)$.
	\begin{enumerate}
		\item If $f \in \C^{\infty}_b(\RR^N;\RR)$, then for both schemes corresponding to $E=E^T$ and $E=E^L$
		\begin{align}
			\label{eq:smoothCubError}
			\sup_{x,y \in \RR^N} \left| \left[P^{E,y}_{t_j,t_{j+1}} - Q^{E,y}_{t_j,t_{j+1}} \right] u^y(t_{j+1},x) \right| & \leq C \, (t_{j+1}-t_j)^{(l+1)/2}.
		\end{align}
		Moreover, the first two derivatives of $u^y$ are bounded:
		\begin{equation}
		\label{eq:smoothDiffGenerators}
		\sup_{(t,y) \in [0,T] \times \RR^N } \left\| u^y(t,\cdot)\right\|_{2, \infty} \leq C .
		\end{equation}
		\item If $f$ is Lipschitz, then
		\begin{align}
			\label{eq:liphCubError}
			\sup_{x,y \in \RR^N} \left| \left[P^{y}_{t_j,t_{j+1}} - Q^{y}_{t_j,t_{j+1}} \right] u^y(t_{j+1},x) \right| & \leq C \, \sum_{k=l}^{l+1} (T-t_{j+1})^{-k/2} \,(t_{j+1}-t_j)^{(k+1)/2}.
		\end{align}
		In addition, the first derivative of $u^y$ is bounded
		\begin{equation}
		\label{eq:Lipu}
		\sup_{(t,x,y) \in [0,T] \times \RR^N \times \RR^N} \left|\nabla u^y(t,x) \right| \leq C 
		\end{equation}
		and we have the estimate on the first two derivatives:
		\begin{equation}
		\label{eq:lipDiffGenerators}
		\sup_{y \in  \RR^N } \left\| u^y(t,\cdot)\right\|_{2, \infty} \leq C \, (T-t)^{-m/2} .
		\end{equation}
		Finally, for both schemes corresponding to $E=E^L$ and $E=E^T$,
		\begin{equation}
		\label{eq:liplaststep}
		\sup_{x,y \in \RR^N} \left|	\left[ P^y_{s,t}- Q_{s,t}^{E,y} \right] f(x) \right| \leq C \, \|f \|_{Lip} |t-s|^{1/2}.
		\end{equation}
	\end{enumerate}
\end{lemma}
\begin{proof}
	We think of $\left(X^{t,x,y}_s\right)_{s \in [t,T]}$ and $\left(^{E} X^{t,x,y}_s\right)_{s \in [t,T]}$ as the solutions of the SDEs with coefficients
	\begin{align*}
		W^y_i(s,z)&:= V_i(z, \EE \varphi_i (X^{0,y}_s)) , \\
		^{E} W^y_i(s,z)&:= V_i(z, E^{y}_s(\varphi_i) ) \quad i=0, \ldots, d,
	\end{align*}
	respectively.
	We think of  $\left\{W^y_i(t,\cdot), \: ^{E} W^y_i(t,\cdot): i=0,\ldots, d \right\}$ as vector fields on $\RR^N$ depending on time $t \in [0,T]$ and the parameter $y \in \RR^N$. In the proof of Lemma \ref{lem:MKVbds}, we explained that $W^y_0, \ldots, W^y_d \in \C^{\infty}_b([0,T]\times \RR^N;\RR^N)$ with all bounds uniform in $y$. The same is true of the functions $^{E} W^y_0 \ldots,^{E}W^y_d$. To see this, we note that for the both schemes, the map $s \mapsto E^{y}_s(\varphi_i)$ is a polynomial, therefore smooth with bounded derivatives on $[0,T]$.
	
	We use the notation  $W^y_{[\alpha]}$ and $^E W^y_{[\alpha]}$ for iterated Lie brackets of the vector fields introduced just before Assumption \ref{ass:MKVcub}.
 	Then, we note that for all $(s,z,y) \in [0,T] \times \RR^N \times \RR^N$ and $\alpha \in \cup_{k \geq1} \{1, \ldots, N\}^k$,
	\begin{align*}
		&  \langle W^y_{[\alpha]}(s,z), \xi \rangle^2 \geq  \inf_{x' \in \RR}  \langle V_{[\alpha]}(z,x'), \xi \rangle^2 ,
	\end{align*}
	so that
	\begin{align*}
		& \inf_{(s,z,y) \in [0,T] \times \RR^N \times \RR^N} \langle W^y_{[\alpha]}(s,z), \xi \rangle^2 \geq  \inf_{(z,x') \in \RR^N \times \RR}  \langle V_{[\alpha]}(z,x'), \xi \rangle^2 .
	\end{align*}
	Hence, under the uniform strong H\"ormander condition (A1),
	\begin{align*}
		&\inf_{(s,z,y) \in [0,T] \times \RR^N \times \RR^N} \sum_{ \alpha \in \cup_{k=1}^m \{1, \ldots, N \}^k } \langle W^y_{[\alpha]}(s,z,y), \xi \rangle^2 
		\geq \delta \, |\xi |^2 ,
	\end{align*}
	so the vector fields $\{W^y_i:i=1, \ldots d\}$ satisfy a uniform strong H\"ormander condition. Exactly the same holds true for the vector fields $\{ ^E W^y_i:i=1, \ldots d\}$. This uniform strong H\"ormander condition is stronger than the $\widetilde{UFG}$ condition, hence, we have the results of Section \ref{sec:hormander} available to us, subject to slight modification since the coefficients in the current setting also depend on a parameter.
	
	Now, for $f \in \C^{\infty}_b(\RR^N;\RR)$, by differentiating under the expectation and using the moment bounds on $\partial^{\eta}_x X^{t,x,y}_s$ contained in \eqref{eq:decMKVderivbd} we see that for all multi-indices $\eta$ on $\{1, \ldots, N\}$ with length at least one,
	\[
	\sup_{x,y \in \RR^N} \left| \partial^{\eta}_x u^y(t,x) \right| < \infty,
	\]
	so \eqref{eq:smoothDiffGenerators} holds. 
The one step cubature error conrtained in \eqref{eq:smoothCubError} follows from a stochastic Taylor expansion, noting
that for all $\beta \in \A_1$,
	\[
	\sup_{x,y \in \RR^N}	\left|\widetilde{ W}^y_{\beta} u^y(t,x) \right| < \infty.
	\]
	This again follows from the boundedness of derivatives of $u^y(t, \cdot)$ and $W^y_i(t, \cdot)$ uniformly in $y$.
	
	For $f$ Lipschitz, the bound in \eqref{eq:liphCubError} is the same as  \eqref{eq:onestepcubappendix}, adapted to the case where coefficients also depend on a parameter.
	The estimate \eqref{eq:Lipu} comes from Corollary \ref{cor:hormanderIBP} adapted to the case where coefficients also depend on a parameter.
	
	Finally, when $f$ is Lipschitz,
	\begin{align*}
		\left|	\left[ P^y_{s,t}- Q_{s,t}^{E,y} \right] f(x) \right| &= \left|	\EE \left[ f(X_t^{s,x,y})\right] - \EE_{\QQ} \left[ f(^EX_t^{s,x,y}) \right] \right| \\
		& \leq  \left|	\EE \left[ f(X_t^{s,x,y})\right] - f(x)  \right| + \left| \EE_{\QQ} \left[ f(^EX_t^{s,x,y})\right] - f(x) \right| \\
		& \leq \|f\|_{Lip} \left(	\EE \left|  X_t^{s,x,y} -x\right|  +	\EE_{\QQ} \left|  ^EX_t^{s,x,y} -x\right|  \right)
	\end{align*}
	That $	\EE \left|  X_t^{s,x,y} -x\right| \leq C \, |t-s|^{1/2}$ is a standard result for SDEs with bounded coefficients. For the other term,
	\[
	\EE_{\QQ} \left| ^E X_t^{s,x,y} -x\right| = \sum_{i=1}^{N_{Cub}} \lambda_i \left|  ^E X_t^{s,x,y,i} -x\right|.
	\] 
	where $^E X_t^{s,x,y,i}$ is the solution of the ODE along the $i$-th cubature path. Then, we have
	\[
	\EE_{\QQ} \left| ^E X_t^{s,x,y} -x\right| \leq C |t-s|,
	\]
	due to a standard estimate on the solution of an ODE with bounded coefficients.
\end{proof}

Before we discuss how accurate the polynomial approximations are, we need a lemma on the concerning the time partitions we use and a type of sum involving its increments which will appear in the error analysis. 
\begin{lemma}
	\label{lem:Kpartition}
\begin{enumerate}
	\item 
	Let $a > b \geq 0$, let $\gamma > \textstyle\frac{a-1}{a-b}$ and let $t_j$ be times in points in the Kusuoka partition, then there is a constant $C=C(\gamma)>0$ such that
	We consider the sum
	
	\begin{equation}
\sum_{j=0}^{n-2} (t_{j+1}-t_j)^a (T-t_{j+1})^{-b} \leq
	C \, n^{-(a-1)} .
	\end{equation}
\item
	For the partition $\Pi^{\gamma,r}_n$, 
	
	\begin{equation}
	\frac{1}{((j+1) \wedge r)!}\prod_{k=0}^{j \wedge (r-1)} (t_{j+1}-t_{j-k})  \leq C \, n^{-r}
	\end{equation}
	and for $a>b \geq 0$ and $\gamma > \textstyle\frac{a-1}{a-b}$ 
	\begin{equation}
	\sum_{j=0}^{n-2} (t_{j+1}-t_j)^a (T-t_{j+1})^{-b}  \leq
	C \, n^{-(a-1)} \, \left(1-r/n \right)^{-(a+b-1)} .
	\end{equation}
\end{enumerate}
\end{lemma}
\begin{proof}
\begin{enumerate}
\item 
	This is proved in a slightly different format in Crisan \& Ghazali \cite{crisanghazali}. First, note that
	\begin{align*}
	t_{j+1}-t_j &= T \left[ \left(1-\frac{j}{n}\right)^{\gamma} - \left(1-\frac{j+1}{n}\right)^{\gamma} \right] \\
	& =T \gamma \int_{\left(1-\frac{j+1}{n}\right)}^{\left(1-\frac{j}{n}\right)}  u^{\gamma-1} \, du \\
	&\leq  T \gamma \frac{1}{n} \left(1-\frac{j}{n}\right)^{\gamma-1} .
	\end{align*}
	Then, we use that $\left(1-\frac{j}{n}\right) \leq 2 \left(1-\frac{j+1}{n}\right)$ for $j \in \{0, \ldots, n-2\}$ to get
	\begin{equation}
	\label{eq:Kuspart}
	t_{j+1}-t_j \leq C \, \frac{1}{n} \left(1- \frac{j+1}{n} \right)^{\gamma-1},
	\end{equation}
	By definition,
	\[
	T-t_{j+1} = T \left( 1-\frac{j+1}{n} \right)^{\gamma},
	\]
	so that 
	\begin{align*}
\sum_{j=0}^{n-2} (t_{j+1}-t_j)^a (T-t_{j+1})^{-b}  \leq C \sum_{j=0}^{n-2} n^{-a} \left(1-\frac{j+1}{n}\right)^{a(\gamma-1)} \left( 1-\frac{j+1}{n} \right)^{-b\gamma}.
	\end{align*}
	Re-ordering the terms, we get
	\begin{align}
	\label{eq:Sabn}
\sum_{j=0}^{n-2} (t_{j+1}-t_j)^a (T-t_{j+1})^{-b}  \leq n^{-(a-1)} \sum_{j=1}^{n-1} n^{-1} \left(\frac{j}{n}\right)^{a(\gamma-1)-b\gamma} .
	\end{align}
	We note that
	\[
	\sum_{j=1}^{n-1} n^{-1} \left(\frac{j}{n}\right)^{a(\gamma-1)-b\gamma} \leq \int_0^1 x^{a(\gamma-1)-b\gamma} \, dx
	\]
	and the condition $\gamma > \frac{a-1}{a-b}$ guarantees that the exponent $a(\gamma-1)-b\gamma >-1$, so that the integral is finite.
	
\item 
	First, note $ t_{j+1}-t_{j-k} = (t_{j+1}-t_j) + (t_j-t_{j-1}) + \ldots +(t_{j-k+1}-t_{j-k})$. There are $(k+1)$ terms in this sum, and, in the case $j <r-1$, from the definition of the first $r$ steps of the partition, the biggest of these is $T n^{-r/(j+1)}$. Hence, $t_{j+1}-t_{j-k} \leq (k+1)Tn^{-r/(j+1)}$. So, in this case,
	\begin{align*}
		\frac{1}{((j+1) \wedge r  )!}\prod_{k=0}^{j \wedge (r-1)} (t_{j+1}-t_{j-k})  &= \frac{1}{ (j+1)!}\prod_{k=0}^{ j} (t_{j+1}-t_{j-k})\\
		& \leq  \frac{1}{(j+1)!}\prod_{k=0}^{j } (k+1) T \, n^{-r/(j+1)} \\
		& = T^j \, n^{-r}.
	\end{align*}
	In the case $j \geq r-1$,  there is a constant $K$ such that for each interval $t_{j+1}-t_{j}\leq K/n$, so that $t_{j+1}-t_{j-k} \leq K (k+1)/n$ and
	\begin{align*}
		\frac{1}{(r \wedge j)!}\prod_{k=0}^{(r \wedge j)-1} (t_{j+1}-t_{j-k})  &= \frac{1}{ r!}\prod_{k=0}^{ r-1} (t_{j+1}-t_{j-k})\\
		& \leq  \frac{1}{r!}\prod_{k=0}^{r-1} K\, (k+1)/n \\
		& = K^{r} \, n^{-r}.
	\end{align*}
	Now considering the sum $\textstyle \sum_{j=0}^{n-2} (t_{j+1}-t_j)^a (T-t_{j+1})^{-b}$, we split it into two parts: when $0 \leq j \leq r-1$, $t_{j+1}-t_j = T n^{-r/(j+1)} \leq C\, n^{-1}$ and
	$
	T-t_{j+1} \geq T  \left(1 - r n^{-1} \right),
	$
	using that $r \leq n/2$. So, 
	\begin{equation}
	\label{eq:halfsum}
	\sum_{j=0}^{r-1} (t_{j+1}-t_j)^a (T-t_{j+1})^{-b} \leq C \, r \, n^{-a} \left(1 - r n^{-1} \right)^{-b} \leq C \, n^{-a} \, \left(1 - r n^{-1} \right)^{-b}
	\end{equation}
	
	For $r \leq j \leq n-2$, the same analysis as in Lemma \ref{lem:Kpartition} gives
	\[
	t_{j+1}-t_j \leq C \, \frac{1}{n} \left(1- \frac{j+1-r}{n-r} \right)^{\gamma-1},
	\]
	By definition,
	\[
	T-t_{j+1} = (T-t_r) \left( 1-\frac{j+1-r}{n-r} \right)^{\gamma} \geq T \left(1-r/n\right) \left( 1-\frac{j+1-r}{n-r} \right)^{\gamma},
	\]
	The proof then follows as in the first part of this lemma to give for $\gamma > \textstyle\frac{a-1}{a-b}$
	\begin{equation*}
		\sum_{j=r}^{n-2} (t_{j+1}-t_j)^a (T-t_{j+1})^{-b}   \leq
		C \,\left(1-r/n\right)^{-b} (n-r)^{-(a-1)} .
	\end{equation*}
	We then note that $\frac{1}{n-r} = \frac{1}{n} . \frac{1}{1-r/n}$ 
	Combining this with \eqref{eq:halfsum} gives the result.
\end{enumerate}
\end{proof}

\begin{lemma}[Polynomial approximations]
	\label{lem:polys}
	For the Taylor approximation, there exists a finite collection of functions $\mathcal{H} \subset \C^{\infty}_b(\RR^N;\RR)$ such that 
	\begin{align}
		\label{eq:Taypolyerror}
		\begin{split}
			\sup_{s \in [t_j,t_{j+1}]} \left| \EE \varphi_i(X^{0,y}_{s}) - E^{T,y}_{s}(\varphi_i) \right| & \leq  C \bigg\{ (t_{j+1}-t_j)^{q+1} \\
			& \quad \quad+   \sum_{\psi \in \mathcal{H} } \left|\left( Q^{E^T,y,\Pi_n}_{t_j}  - P^{y}_{0,t_j} \right)\psi(y)  \right| \bigg\}.
		\end{split}
	\end{align}
For the Lagrange interpolation method:
\begin{align}
	\label{eq:Lagpolyerror}
	\begin{split}
		\sup_{s \in [t_j,t_{j+1}]} \left| \EE \varphi_i(X^{0,y}_{s}) - E^{y}_{s}(\varphi_i) \right|
		\leq C \, \bigg\{ & \frac{1}{((j+1) \wedge r)!}\prod_{k=0}^{j \wedge (r-1)} (t_{j+1}-t_{j-k}) \\
		& + \sum_{k=0}^{j \wedge (r-1)} \left| \left( Q^{E,y,\Pi_n}_{t_{j-k}}  - P^{y}_{0,t_{j-k}} \right) \varphi_i  \right| \bigg\}
	\end{split}
\end{align}
\end{lemma}

\begin{proof}
	
	\underline{Taylor method:}
	Now we recall, for the Taylor method with $s \in [t_j,t_{j+1}]$,
	\[
	E^{T,y}_{s}(\varphi_i) := \sum_{k=0}^q \frac{1}{k!} \, (	\mathcal{T}^{Q}_{t_j})^k (\EE \varphi_i(X^{0,y})) \, (s-t_j)^k.
	\]
	We will estimate the error $\left| \EE \varphi_i(X^{0,y}_{s}) - E^{T,y}_{s}(\varphi_i) \right|$ by splitting it into
	\begin{equation}
	\label{eq:taylore}
	\left| \EE \varphi_i(X^{0,y}_{s}) - E^{T,y}_{s}(\varphi_i) \right| \leq \left| \EE \varphi_i(X^{0,y}_{s}) - \widehat{E}^{y}_{s}(\varphi_i) \right| + \left|\widehat{ E}^{y}_{s}(\varphi_i) - E^{T,y}_{s}(\varphi_i) \right|, 
	\end{equation}
	where 
	\[
	\widehat{E}^{y}_{s}(\varphi_i) := \sum_{k=0}^q \frac{1}{k!} \, (	\mathcal{T}_{t_j})^k (\EE \varphi_i(X^{0,y}))\, (s-t_j)^k, \quad s \in [t_j,t_{j+1})
	\]
	is the truncated Taylor expansion of $s \mapsto \EE \varphi_i(X^{0,y}_{s})$ of order $q$ around $t_j$. It is straightforward that
	\begin{align*}
		\left| \EE \varphi_i(X^{0,y}_{s}) - \widehat{E}^{y}_{s}(\varphi_i) \right| \leq C \, (s-t_j)^{q+1},
	\end{align*}
	and
	\begin{align}
		\label{eq:taylorcube}
		\left| \widehat{E}^{y}_{s}(\varphi_i) -  	E^{T,y}_{s}(\varphi_i) \right| \leq  \sum_{k=0}^q \frac{1}{k!} (s-t_j)^k \, \left| \left[(	\mathcal{T}_{t_j})^k - (\mathcal{T}^{Q}_{t_j})^k \right](\EE \varphi_i(X^{0,y}))  \right|.
	\end{align}
	Now, recall that $	\mathcal{T}_{t_j}^k (\EE \varphi_i(X^{0,x})) \in \bar{\mathcal{S}}$, so it can be written as a sum of products of terms of the form
	\[
	\EE  \, g  \left(X_{t_j}^{0,x}, \EE \varphi(X_{t_j}^{0,x}) \right) ,
	\]
	for some $g \in \C^{\infty}_b(\RR^N \times \RR;\RR)$, and $\varphi \in \C^{\infty}_b(\RR^N;\RR)$. For this type of term, the error in replacing expectations under $\PP$ with expectations under $\QQ^{\Pi_n}$ can be bounded by
	\begin{align*}
		&\left|	\EE  \, g  \left(X_{t_j}^{0,x}, \EE \varphi(X_{t_j}^{0,x}) \right) - 	\EE_{\QQ^{\Pi_n}}  \, g  \left(X_{t_j}^{0,x}, \EE_{\QQ^{\Pi_n}} \varphi(X_{t_j}^{0,x}) \right) \right| \\
		\leq & \left| 	 \EE \varphi(X_{t_j}^{0,x})  - 	  \EE_{\QQ^{\Pi_n}} \varphi(X_{t_j}^{0,x})  \right| +  \left|	\EE  \, g  \left(X_{t_j}^{0,x}, \EE_{\QQ^{\Pi_n}} \varphi(X_{t_j}^{0,x})  \right) - 	\EE_{\QQ^{\Pi_n}}  \, g  \left(X_{t_j}^{0,x}, \EE_{\QQ^{\Pi_n}} \varphi(X_{t_j}^{0,x})  \right) \right| 
	\end{align*}
	Due to the form of $	\mathcal{T}^k (\EE \varphi_i(X^{0,x})) \in \bar{\mathcal{S}}$, 
	the error $\left[(	\mathcal{T}_{t_j})^k - (\mathcal{T}^{Q}_{t_j})^k \right](\EE \varphi_i(X^{0,y}))$ 
	can be decomposed as a sum of products of such errors for different functions $g$ and $\varphi$.
	We define $\mathcal{H}_{i,k}$ to be the collection of these functions appearing in the expression for $\mathcal{T}^k(\EE \varphi_i(X^{0,y}))$. 
	Then, the error
	\[
	\left| \left[(	\mathcal{T}_{t_j})^k - (\mathcal{T}^{Q}_{t_j})^k \right](\EE \varphi_i(X^{0,y}))  \right|
	\]
	can be bounded by a constant multiple of 
	\[
	\sum_{\psi \in \mathcal{H}_{i,k}}	\left|\left( \EE_{\QQ^{\Pi_n}}  - \EE \right) \psi(X^{0,y}_{t_j})  \right| = \sum_{\psi \in \mathcal{H}_{i,k} }\left|\left( Q^{E^T,y,\Pi_n}_{t_j}  - P^{y}_{0,t_j} \right)\psi(y)  \right|.
	\]
	So, 
	\begin{align*}
		&\sup_{s \in [t_j,t_{j+1}]}	\left| \widehat{E}^{y}_{s}(\varphi_i) -  	E^{T,y}_{s}(\varphi_i) \right|
		\leq  C \, \sum_{k=0}^q \frac{1}{k!} (t_{j+1}-t_j)^k  \sum_{i=0}^d \sum_{\psi \in \mathcal{H}_{i,k} } \left|\left( Q^{E^T,y,\Pi_n}_{t_j}  - P^{y}_{0,t_j} \right)\psi(y)  \right|,
	\end{align*}
	and the largest term in the outer sum on the right hand side occurs when $k=0$, so using this and defining $\mathcal{H}:=\cup_{i=1}^d \mathcal{H}_{i,0}$, the estimate \eqref{eq:taylore} becomes
	\begin{align*}
		\sup_{s \in [t_j,t_{j+1}]} \left| \EE \varphi_i(X^{0,y}_{s}) - E^{T,y}_{s}(\varphi_i) \right| & \leq  C \bigg\{ (t_{j+1}-t_j)^{q+1} \\
		& \quad \quad+ q  \sum_{\psi \in \mathcal{H} } \left|\left( Q^{E^T,y,\Pi_n}_{t_j}  - P^{y}_{0,t_j} \right)\psi(y)  \right| \bigg\}.
	\end{align*}
	
	\underline{Lagrange method}
	
		Let $z \in C^{k+1}([0,T]; \RR)$. Recall that we denote by $L[(t_1,x_1), \ldots, (t_k,x_k)]$ the Lagrange Interpolating Polynomial passing through the points $\left(t_1,x_1 \right), \ldots, \left(t_k,x_k \right)$. The error in approximating $z(t)$ with  the polynomial $L[(t_1,z(t_1)), \ldots, (t_k,z(t_k))]$ for any $t \in [0,T]$ is
		\begin{equation}
		\label{eq:polyerror}
		z(t)- L[(t_1,z(t_1)), \ldots, (t_k,z(t_k))] (t) = \frac{1}{k!} \, z^{(k)} (\xi) \, \prod_{j=1}^k (t-t_j),
		\end{equation}
		where $\xi$ is some point in $[0,T]$.	
		Note also that we can write $L[(t_1,x_1), \ldots, (t_k,x_k)]$ terms of the Lagrange basis polynomials $L_j :[0,T] \to \RR$ as
		\[
		L[(t_1,x_1), \ldots, (t_k,x_k)] (t) = \sum_{j=1}^k x_j \, L_j(t),
		\]
		where 
		\[
		L_j(t) = \prod_{i=1, i \neq j}^k \frac{t-t_i}{t_j-t_i}
		\]
		So, the difference between polynomilas interpolating different points on the same time grid is given by
		\[
		L[(t_1,x_1), \ldots, (t_k,x_k)] (t) - L[(t_1,y_1), \ldots, (t_k,y_k)] (t) = \sum_{j=0}^k (x_j-y_j) \, L_j(t).
		\]
		and, in particular,
		\begin{equation}
		\label{eq:polydiff}
		\sup_{t \in [0,T]} \left|L^{k}[x_1, \ldots, x_k] (t) - L^{k}[y_1, \ldots, y_k] (t) \right| \leq C(T) \sum_{j=1}^k  \left| x_j-y_j \right|.
		\end{equation}
		Now, recall the definition of the Lagrange interpolation approximation is
		\[
		E^{L,y}_t(\varphi_i): =  L \left[ \left\{ \left(t_{j-k} ,\EE^{\QQ^{\Pi_n}} \varphi_i \left (X^{0,y}_{t_{j-k}} \right)  \right) \right\}_{k=0}^{j \wedge (r-1)} \right]  (t), \quad \quad t \in [t_j,t_{j+1}],
		\]
		and consider the same object but with all expecations under the Wiener measure, $\PP$:
		\[
		E^{Ly,\PP}_t(\varphi_i): =  L \left[ \left\{ \left(t_{j-k} ,\EE \varphi_i \left (X^{0,y}_{t_{j-k}} \right)  \right) \right\}_{k=0}^{j \wedge (r-1)} \right]  (t), \quad \quad t \in [t_j,t_{j+1}].
		\]
		Then, we can split the error  error $\EE \varphi_i(X_t^{0,y}) - E^y_t(\varphi_i)$ into 
		\[
		\left[	\EE \varphi_i(X_t^{0,y}) - E^{L,y,\PP}_t(\varphi_i) \right] + \left[E^{Ly,\PP}_t(\varphi_i) - E^{L,y}_t(\varphi_i)\right]
		\]
		We are able to control using the first term suing \eqref{eq:polyerror} and the second term  using \eqref{eq:polydiff}. The result follows immediately.

\end{proof}

\section{Proof of Theorem \ref{th:MKVCub}}
\begin{remark}
	\label{rem:expl}
	For the Taylor method, our proof is essentially the same as that in \cite{nicecubature}, when $ f \in \C^{\infty}_b(\RR^N;\RR)$. When $f$ is Lipschitz, however, we split the local errors differently. In \cite{nicecubature}, the error is split into
	\begin{align}
		\left[ P^y_{t_j,t_{j+1}}- Q_{{t_j,t_{j+1}}}^{E,y} \right]u^y  = \left[ P_{{t_j,t_{j+1}}}^{y}- P_{{t_j,t_{j+1}}}^{E,y} \right] u^y
		+ \left[ P^{E,y}_{t_j,t_{j+1}}- Q_{{t_j,t_{j+1}}}^{E,y} \right]u^y .
	\end{align}
	The term $\left[ P_{{t_j,t_{j+1}}}^{y}- P_{{t_j,t_{j+1}}}^{E,y} \right]u^y $ is then estimated in terms of the difference of the generators of the processes $X$ and $^E X$ applied to $u^y$, which in turn depends on an estimate on $\nabla^2 u^y$. In the uniformly elliptic case with $f$ Lipschitz, in \cite{nicecubature}, $|\nabla^2 u^y(t,x)| \leq C (T-t)^{-1/2}$. However, in the H\"ormander case, $|\nabla^2 u^y(t,x)| \leq C (T-t)^{-m/2}$, where $m$ is the order of the H\"ormander condition, which could be very large. Instead, here, we split the error into
	\begin{align}
		\left[ P^y_{t_j,t_{j+1}}- Q_{{t_j,t_{j+1}}}^{E,y} \right]u^y  &= \left[ Q_{{t_j,t_{j+1}}}^{y}- Q_{{t_j,t_{j+1}}}^{E,y} \right]u^y   + \left[ P^{y}_{t_j,t_{j+1}}- Q_{{t_j,t_{j+1}}}^{y} \right]u^y.
	\end{align} 
	We control the term $\left[ Q_{{t_j,t_{j+1}}}^{y}- Q_{{t_j,t_{j+1}}}^{E,y} \right]u^y$ using only the Lipschitz constant of $u^y$ which is uniformly bounded in time.
\end{remark}

\subsection{Smooth bounded terminal condition}

We introduce the generator associated to the process $X^{0,y}$
\begin{align*}
	& \mathcal{L}_s^y: = V_0(\cdot, \EE \varphi_0(X_s^{0,y}))  + \frac{1}{2} \sum_{i=1}^d V_i(\cdot, \EE \varphi_i(X_s^{0,y}))^2 ,
\end{align*}
and note that
$u^y(t,x)=P^y_{t,T}f(x)$ solves the PDE
\begin{align}
	\label{eq:scalarPDE}
	\begin{split}
		\left( \partial_t + \mathcal{L}_t^y \right) u^y(t,x) &= 0, \\
		u^y(T,x)&=f(x).
	\end{split}
\end{align}
In the analysis of each scheme, we split the local error into
\begin{align}
	\label{eq:Eerror}
	\left[ P^y_{t_j,t_{j+1}}- Q_{{t_j,t_{j+1}}}^{E,y} \right] u^y(t_{j+1},x) &= \left[ P^y_{t_j,t_{j+1}}- P_{{t_j,t_{j+1}}}^{E,y} \right] u^y(t_{j+1},x) \\
	\label{eq:Cuberror}	& \quad \quad + \left[ P^{E,y}_{t_j,t_{j+1}}- Q_{{t_j,t_{j+1}}}^{E,y} \right] u^y(t_{j+1},x).
\end{align}
Equation \eqref{eq:Eerror} is the error due to approximating the $ \EE \varphi_i(X_t^{0,y})$ by $E^y_t(\varphi_i)$, and \eqref{eq:Cuberror} is a one-step cubature error. Now,
\begin{align*}
	\left[ P^y_{t_j,t_{j+1}}- P_{{t_j,t_{j+1}}}^{E,y} \right] u^y(t_{j+1},x) &= \EE \left[ u^y(t_{j+1},X_{t_{j+1}}^{t_j,x,y}) - u^y(t_{j+1},^EX_{t_{j+1}}^{t_j,x,y}) \right] \\
	&=  \EE \int_{t_j}^{t_{j+1}} \left( \mathcal{L}_s^y - \mathcal{L}_s^{E,y} \right) u^y(s,^EX_{s}^{t_j,x,y}) \, ds
\end{align*}
Using the Lipschitz property of the coefficients, we get
\[
\left| \left( \mathcal{L}_s^y - \mathcal{L}_s^{E,y} \right) u^y(s,^EX_{s}^{t_j,x,y})  \right| \leq C \,	\|u^y(s,\cdot)\|_{2, \infty} \, \sum_{i=0}^d \left| \EE \varphi_i(X^{0,y}_{s}) - E^{y}_{s}(\varphi_i) \right|.
\]
Now, we recall from Lemma \ref{lem:MKCcubreg} that $\|u^y(s,\cdot)\|_{2, \infty} \leq C$ when $f \in \C^{\infty}_b(\RR^N;\RR)$, so
\begin{align}
	\label{eq:genLocalerror}
	\left[ P^y_{t_j,t_{j+1}}- P_{{t_j,t_{j+1}}}^{E,y} \right] u^y(t_{j+1},x) \leq C \int_{t_{j}}^{t_{j+1}}  \sum_{i=0}^d \left| \EE \varphi_i(X^{0,y}_{s}) - E^{y}_{s}(\varphi_i) \right| \, ds
\end{align}

Now, to control the right hand side above, we use Lemma \ref{lem:polys} and we split the proof depending on the individual scheme.

\subsubsection{Taylor Method}

Using Lemma \ref{lem:polys}, for the Taylor method,  \eqref{eq:genLocalerror} becomes
\begin{align}
	\begin{split}
		\left[ P^y_{t_j,t_{j+1}}- P_{{t_j,t_{j+1}}}^{E^T,y} \right] u^y(t_{j+1},x) \leq C (t_{j+1} - t_{j}) \bigg\{ &(t_{j+1}-t_j)^{q+1} \\
		& \quad \quad+  \sum_{\psi \in \mathcal{H} } \left|\left( Q^{E^T,y,\Pi_n}_{t_j}  - P^{y}_{0,t_j} \right)\psi(y)  \right| \bigg\}.
	\end{split}
\end{align}
Summing up the local errors, the global error is then given by
\begin{align}
	\begin{split}
		\left[ P^y_{0,t_{n}}- Q_{t_{n}}^{E^T,y,\Pi_n} \right] f(x) & \leq C \sum_{j=0}^{n-1} \bigg\{ (t_{j+1}-t_j)^{q+2} \\
		& + q (t_{j+1}-t_j) \sum_{\psi \in \mathcal{H} } \left|\left( Q^{E^T,y,\Pi_n}_{t_j}  - P^{y}_{0,t_j} \right)\psi(y)  \right| + (t_{j+1}-t_j)^{(l+1)/2}\bigg\}.
	\end{split}
\end{align}
The above holds for any $f \in \C^{\infty}_b(\RR^N;\RR)$. In particular, we can take $f= \psi$ for any $\psi \in \mathcal{H}$. Doing this and repeatedly applying the discrete Gronwall inequality, 
\begin{align}
	\label{eq:Taylorsmoothglobal}
	\begin{split}
		\left[ P^y_{0,t_{n}}- Q_{{0,t_{n}}}^{E,y} \right] f(x) & \leq C \exp \left(  q \sum_{j=0}^{n-1} (t_{j+1}-t_j) \right)  \\
		&\times  \sum_{j=0}^{n-1} \bigg\{ (t_{j+1}-t_j)^{q+2}  + (t_{j+1}-t_j)^{(l+1)/2}\bigg\}.
	\end{split}
\end{align}

\subsubsection{Lagrange interpolation method}

Using Lemma \ref{lem:polys}, for the Lagrange interpolation method,  \eqref{eq:genLocalerror} becomes
\begin{align}
	\begin{split}
		\left[ P^y_{t_j,t_{j+1}}- P_{{t_j,t_{j+1}}}^{E^L,y} \right] u^y(t_{j+1},x) \leq C (t_{j+1} - t_{j}) \bigg\{ & \frac{1}{((j+1) \wedge r)!}\prod_{k=0}^{j \wedge (r-1)} (t_{j+1}-t_{j-k}) \\
		& + \sum_{k=0}^{j \wedge (r-1)} \left| \left( Q^{E^T,y,\Pi_n}_{t_{j-k}}  - P^{y}_{0,t_{j-k}} \right) \varphi_i  \right| \bigg\}.
	\end{split}
\end{align}
The global error is then given by
\begin{align}
	\begin{split}
		\left[ P^y_{0,t_{n}}- Q_{{0,t_{n}}}^{E^L,y} \right] f(x) \leq C \sum_{j=0}^{n-1} \bigg[ & (t_{j+1} - t_{j}) \bigg\{  \frac{1}{((j+1) \wedge r)!}\prod_{k=0}^{j \wedge (r-1)} (t_{j+1}-t_{j-k}) \\
		&+ \sum_{k=0}^{j \wedge (r-1)} \left| \left( Q^{E^T,y,\Pi_n}_{t_{j-k}}  - P^{y}_{0,t_{j-k}} \right) \varphi_i  \right| \bigg\} +  (t_{j+1}-t_j)^{(l+1)/2}\bigg].
	\end{split}
\end{align}
Taking $f= \varphi_i$ for any $i=0, \ldots, d$ and using discrete Gronwall inequality, we obtain
\begin{align}
	\label{eq:Lagglobalsmooth}
	\begin{split}
		\left[ P^y_{0,t_{n}}- Q_{{0,t_{n}}}^{E^L,y} \right] f(x)  \leq C & \exp \left(  r \sum_{j=0}^{n-1} (t_{j+1}-t_j) \right) \\
		&\times \sum_{j=0}^{n-1} \bigg\{  \frac{(t_{j+1}-t_j)}{((j+1) \wedge r)!}\prod_{k=0}^{j \wedge (r-1)} (t_{j+1}-t_{j-k}) 
		+ (t_{j+1}-t_j)^{(l+1)/2} \bigg\}.
	\end{split}
\end{align}

\subsection{Lipschitz terminal condition, $m = 1$}

In this case, the estimate we have on the first two derivatives of $u^y$ is $\|u^y(t, \cdot)\|_{2,\infty} \leq C (T-t)^{-1/2}$. Using this estimate we get, similarly to \eqref{eq:genLocalerror},
\begin{align}
	\label{eq:genlocalerrorm1}
	\left[ P^y_{t_j,t_{j+1}}- P_{{t_j,t_{j+1}}}^{E,y} \right] u^y(t_{j+1},x) \leq C\, (T-t_{j+1})^{-1/2} \int_{t_{j}}^{t_{j+1}}  \sum_{i=0}^d \left| \EE \varphi_i(X^{0,y}_{s}) - E^{y}_{s}(\varphi_i) \right| \, ds
\end{align}

\subsubsection{Taylor method}

The same arguments as the previous section give the global error is
\begin{align}
	\begin{split}
		\left[ P^y_{0,t_{n}}- Q_{t_{n}}^{E^T,y,\Pi_n} \right] f(x) & \leq C  \sum_{j=0}^{n-2} (T-t_{j+1})^{-1/2} \bigg\{ (t_{j+1}-t_j)^{q+2} \\
		& + (t_{j+1}-t_j) \sum_{\psi \in \mathcal{H}} \left|\left( Q^{E^T,y,\Pi_n}_{t_j}  - P^{y}_{0,t_j} \right)\psi(y)  \right| + (t_{j+1}-t_j)^{(l+1)/2}\bigg\} \\
		& + \left[ P^y_{t_{n-1},t_{n}}- Q_{t_{n-1}, t_{n}}^{E^T,y} \right] f(x).
	\end{split}
\end{align}
Since $\psi \in \C^{\infty}_b(\RR^N;\RR)$, in particular it is Lipschitz. The above estimate holds for all Lipschitz $f$, so taking $f=\psi \in \mathcal{H}$ and using the discrete Gronwall inequality, we get
\begin{align}
	\label{eq:Taylorlipglobalm1}
	\begin{split}
		\left[ P^y_{0,t_{n}}- Q_{t_{n}}^{E^T,y,\Pi_n} \right] f(x) & \leq C \exp \left( \sum_{j=0}^{n-2} (T-t_{j+1})^{-1/2} (t_{j+1}-t_j) \right)  \\
		& \times  \sum_{j=0}^{n-2} (T-t_{j+1})^{-1/2} \left[ (t_{j+1}-t_j)^{q+2} 
		+ (t_{j+1}-t_j)^{(l+3)/2}\right] \\
		& + \left[ P^y_{t_{n-1},t_{n}}- Q_{t_{n-1}, t_{n}}^{E^T,y} \right] f(x).
	\end{split}
\end{align}
Now, we recall that in this setting we use the Kusuoka partition $\Pi^{\gamma}_n$ with $\gamma>l-1$. Using Lemma \ref{lem:Kpartition}, we can see that $\textstyle \sum_{j=0}^{n-1} (T-t_{j+1})^{-1/2} (t_{j+1}-t_j)$ is bounded independently of $n$. For the other two sums, we also use Lemma \ref{lem:Kpartition} and for the final term we use \eqref{eq:liplaststep} to get
\begin{align*}
	\begin{split}
		\left[ P^y_{0,t_{n}}- Q_{t_{n}}^{E^T,y,\Pi_n} \right] f(x) & \leq C  \left( n^{-(l-1)/2} + n^{-(q+1)} + n^{-\gamma/2}  \right) .
	\end{split}
\end{align*}
Noting $\gamma>l-1$ gives the result \eqref{eq:errLipTaylorm1}.

\subsubsection{Lagrange interpolation method}

The same arguments as the previous section give the global error is
\begin{align}
	\begin{split}
		\left[ P^y_{0,t_{n}}- Q_{{0,t_{n}}}^{E^L,y} \right] f(x) \leq C &\sum_{j=0}^{n-2} (T-t_{j+1})^{-1/2} \bigg[  (t_{j+1} - t_{j}) \bigg\{  \frac{1}{((j+1) \wedge r)!}\prod_{k=0}^{j \wedge (r-1)} (t_{j+1}-t_{j-k}) \\
		&+ \sum_{k=0}^{(r \wedge j)-1} \left| \left( Q^{E^L,y,\Pi_n}_{t_{j-k}}  - P^{y}_{0,t_{j-k}} \right) \varphi_i  \right| \bigg\} +  (t_{j+1}-t_j)^{(l+1)/2}\bigg] \\
		& + \left[ P^y_{t_{n-1},t_{n}}- Q_{t_{n-1}, t_{n}}^{E^L,y} \right] f(x).
	\end{split}
\end{align}
Since $\varphi_i \in \C^{\infty}_b(\RR^N;\RR)$, in particular it is Lipschitz. The above estimate holds for all Lipschitz $f$, so taking $f=\varphi_i$, $i=0, \ldots, d$ and using the discrete Gronwall inequality, we get
\begin{align}
	\label{eq:Laglipglobalm1}
	\begin{split}
		\left[ P^y_{0,t_{n}}- Q_{t_{n}}^{E^L,y,\Pi_n} \right] f(x) & \leq C \exp \left(r \, \sum_{j=0}^{n-2} (T-t_{j+1})^{-1/2} (t_{j+1}-t_j) \right)  \\
		& \times  \sum_{j=0}^{n-2} (T-t_{j+1})^{-1/2} \left[ \frac{1}{((j+1) \wedge r)!}\prod_{k=0}^{j \wedge (r-1)} (t_{j+1}-t_{j-k})
		+ (t_{j+1}-t_j)^{(l+1)/2}\right] \\
		& + \left[ P^y_{t_{n-1},t_{n}}- Q_{t_{n-1}, t_{n}}^{E^L,y} \right] f(x).
	\end{split}
\end{align}
By part 1 Lemma \ref{lem:Kpartition},
\[
	\sum_{j=0}^{n-2} (T-t_{j+1})^{-1/2} \, (t_{j+1}-t_j)^{(l+1)/2} \leq C \, n^{-(l-1)/2} \, (1-r/n)^{-l/2}.
\]
Now, we recall that in this setting we use the modified Kusuoka partition $\Pi^{\gamma,r}_n$ with $\gamma>l-1$.
We note
\begin{align*}
	&\sum_{j=0}^{n-2} (T-t_{j+1})^{-1/2}  \frac{1}{((j+1) \wedge r)!}\prod_{k=0}^{j \wedge (r-1)} (t_{j+1}-t_{j-k}) \\
	= & \sum_{j=0}^{n-2} (T-t_{j+1})^{-1/2}(t_{j+1}-t_{j})  \times \left[ \frac{1}{((j+1) \wedge r)!}\prod_{k=1}^{j \wedge (r-1)} (t_{j+1}-t_{j-k}) \right]
\end{align*}
with the left hand term in the product above being bounded uniformly in $n$ by part 1 Lemma \ref{lem:Kpartition} and the second term being less than $n^{-(r-1)}$ by part 2 of the same lemma.
For the final term in \eqref{eq:Laglipglobalm1}, we use \eqref{eq:liplaststep} to get
\[
	\left[ P^y_{t_{n-1},t_{n}}- Q_{t_{n-1}, t_{n}}^{E^L,y} \right] f(x) \leq C \, (n-r)^{-\gamma/2},
\]
then noting that $(n-r)^{-\gamma/2}= n^{-\gamma/2}\,(1-r/n)^{-\gamma/2} \leq n^{-(l-1)/2}\,(1-r/n)^{-l/2}$ when $ \gamma \in (l-1,l)$,
we finally obtain
\begin{align*}
	\begin{split}
		\left[ P^y_{0,t_{n}}- Q_{t_{n}}^{E^L,y,\Pi^{\gamma,r}_n} \right] f(x) & \leq C  \left( n^{-(l-1)/2} (1-r/n)^{-l/2}+ n^{-(r-1)}  \right) .
	\end{split}
\end{align*}
This proves the result \eqref{eq:errLipLagm1}.

\subsection{Lipschitz terminal condition, $m \geq 2$}

When $f$ is Lipschitz and $m \geq 2$, we split the local error into
\begin{align}
	\label{eq:EerrorLip}
	\left[ P^y_{t_j,t_{j+1}}- Q_{{t_j,t_{j+1}}}^{E,y,\Pi_n} \right] u(t_{j+1},x) &= \left[ Q_{{t_j,t_{j+1}}}^{y,\Pi_n}- Q_{{t_j,t_{j+1}}}^{E,y,\Pi_n} \right] u(t_{j+1},x) \\
	\label{eq:CuberrorLip}	& \quad \quad + \left[ P^{y}_{t_j,t_{j+1}}- Q_{{t_j,t_{j+1}}}^{y,\Pi_n} \right] u(t_{j+1},x).
\end{align}
Equation \eqref{eq:EerrorLip} is the error due to approximating the $ \EE \varphi_i(X_t^{0,y})$ by $E^y_t(\varphi_i)$, and \eqref{eq:CuberrorLip} is a one-step cubature error. For the term in \eqref{eq:EerrorLip}, we note that
\begin{align}
	\label{eq:QsplitE}
	\begin{split}
		\left| \left[ Q^{y,\Pi_n}_{t_j,t_{j+1}}- Q_{{t_j,t_{j+1}}}^{E,y,\Pi_n} \right] u^y(t_{j+1},x) \right|  = & \, \EE_{\QQ^{\Pi_n}} \left| u^y\left(t_{j+1}, X^{t_j,x,y}_{t_{j+1}
		}\right) - u^y\left(t_{j+1}, ^{E}X^{t_j,x,y}_{t_{j+1}}\right) \right| \\
		\leq  & \, \|\nabla u^y(t_{j+1},\cdot)\|_{\infty} \, \EE_{\QQ^{\Pi_n}} \left|  X^{t_j,x,y}_{t_{j+1}} - 
		^{E}X^{t_j,x,y}_{t_{j+1}} \right| .
	\end{split}
\end{align}
Now, using the Lipschitz property of the coefficients, we note that
\begin{align*}
	& \EE_{\QQ^{\Pi_n}} \left|  X^{t_j,x,y}_{t_{j+1}} -  
	^{E}X^{t_j,x,y}_{t_{j+1}} \right|\\
	\leq & \, \sum_{k=1}^{N_{Cub}} \lambda_k \, \sum_{i=0}^d \int_{t_j}^{t_{j+1}} \left|V_i(X^{t_j,x,y}_s (\omega_k),\EE\varphi_i(X_s^{0,y}) - V_i( ^{E}X^{t_j,x,y}_s (\omega_k),E_s(\varphi_i)) \right| \,  d \omega^i_k(t_j,t_{j+1})(s).
\end{align*}
We recall the re-scaled path $\omega^i_k(t_j,t_{j+1})(s) = \sqrt{t_{j+1}-t_j} \, \omega^i_k \textstyle \left(\frac{s-t_j}{t_{j+1}-t_j}\right) $, so that, under the assumption that $\omega_k$ is absolutely continuous,
\[
\sup_{s \in [t_j, t_{j+1}]} \left| \frac{d}{ds} \omega_k(t_j,t_{j+1})(s) \right| \leq \frac{1}{\sqrt{t_{j+1}-t_j}} \sup_{s \in [0,1]} \left| \frac{d}{ds}\omega_k(s) \right|.
\] 
So, there exists a constant $C$ which depends on $ \textstyle \max_{k =1, \ldots, N_{Cub}} \sup_{s \in [0,1]}| \frac{d}{ds}  \omega_k(s) |$ such that
\begin{align*} 
	& \EE_{\QQ^{\Pi_n}} \left|  X^{t_j,x,y}_{t_{j+1}} -  
	^{E}X^{t_j,x,y}_{t_{j+1}} \right|\\
	\leq & \,  C \, \frac{1}{\sqrt{t_{j+1}-t_j}} \sum_{i=0}^d \int_{t_j}^{t_{j+1}}  \EE_{\QQ^{\Pi_n}} \left|X^{t_j,x,y}_s -  ^{E}X^{t_j,x,y}_s \right| + \left|\EE\varphi_i(X_s^{0,y})- E_s(\varphi_i) \right| \, ds.
\end{align*}
Then, using Gronwall's inequality, we have
\begin{align*}
	& \EE_{\QQ^{\Pi_n}} \left|  X^{t_j,x,y}_{t_{j+1}} - 
	^{E}X^{t_j,x,y}_{t_{j+1}} \right| \leq C \, \sqrt{t_{j+1}-t_j} \,\sup_{s \in [t_j, t_{j+1}]} \left|\EE\varphi_i(X_s^{0,y})- E_s(\varphi_i) \right|
\end{align*}
Now, going back to \eqref{eq:QsplitE}, we have
\begin{align}
	\label{eq:cubdiffsplitE}
	\begin{split}
		&\left|\left[ Q^{y,\Pi_n}_{t_j,t_{j+1}}- Q_{{t_j,t_{j+1}}}^{E,y,\Pi_n} \right] u^y(t_{j+1},x) \right| 
		\leq  C \, \sqrt{t_{j+1}-t_j}  \, \sup_{s \in [t_j, t_{j+1}]} \left|\EE\varphi_i(X_s^{0,y})- E_s(\varphi_i) \right|.
	\end{split}
\end{align}
\noindent From this point on the arguments depend on the individual scheme.

\subsubsection{Taylor method}

Using Lemma \ref{lem:polys}, \eqref{eq:cubdiffsplitE} becomes
\begin{align}
	\begin{split}
		\left|\left[ Q^{y,\Pi_n}_{t_j,t_{j+1}}- Q_{{t_j,t_{j+1}}}^{E^T,y,\Pi_n} \right] u^y(t_{j+1},x) \right|  \leq C (t_{j+1} - t_{j})^{1/2} \bigg\{ &(t_{j+1}-t_j)^{q+1} \\
		& +  \sum_{\psi \in \mathcal{H}} \left|\left( Q^{E^T,y,\Pi_n}_{t_j}  - P^{y}_{0,t_j} \right)\psi(y)  \right| \bigg\}.
	\end{split}
\end{align}
Since $\psi \in \C^{\infty}_b(\RR^N;\RR)$, we can use the global error from the last section for smooth terminal conditions contained in \eqref{eq:Taylorsmoothglobal} to obtain
\begin{align}
	\begin{split}
		\left|\left[ Q^{y,\Pi_n}_{t_j,t_{j+1}}- Q_{{t_j,t_{j+1}}}^{E^T,y,\Pi_n} \right] u^y(t_{j+1},x) \right|  \leq C (t_{j+1} - t_{j})^{1/2} & \bigg\{ (t_{j+1}-t_j)^{q+1} \\
		& +\sum_{i=0}^{j} \left[ (t_{i+1}-t_i)^{q+2}  + (t_{i+1}-t_i)^{(l+1)/2}\right] \bigg\}.
	\end{split}
\end{align}
Now, we simply use that $t_{j+1}-t_j \leq C/n$ to get:
\begin{align}
	\begin{split}
		\left|\left[ Q^{y,\Pi_n}_{t_j,t_{j+1}}- Q_{{t_j,t_{j+1}}}^{E^T,y,\Pi_n} \right] u^y(t_{j+1},x) \right| 
		& \leq   C n^{-1/2} \left\{ n^{-(q+1)} +  n \left\{ n^{-(q+2)}  + n^{-(l+1)/2} \right\} \right\} \\
		& \leq C \left( n^{-(q+3/2)} + n^{-l/2} \right)
	\end{split}
\end{align}
Combining with the local cubature errors and summing up, we get the global error:
\[
\left[ P^y_{0,t_{n}}- Q_{{t_{n}}}^{E^T,y,\Pi_n} \right] f(x)  \leq C \left(	n^{-(q+1/2)} + n^{-(l-2)/2} \right).
\]

\subsubsection{Lagrange interpolation method}

Now, we only consider the modified Kusuoka partition $\Pi_n^{\gamma,r}$. Using Lemma \ref{lem:Kpartition} part 2 and Lemma \ref{lem:polys}, \eqref{eq:cubdiffsplitE} becomes
\begin{align}
	\begin{split}
		\left|\left[ Q^{y}_{t_j,t_{j+1}}- Q_{{t_j,t_{j+1}}}^{E^L,y} \right] u^y(t_{j+1},x) \right| 
		\leq   C (t_{j+1} - t_{j})^{1/2} \left\{ n^{-r} + \sum_{k=0}^{j \wedge (r-1)} \left| \left( Q^{E^L,y,\Pi_n}_{t_{j-k}}  - P^{y}_{0,t_{j-k}} \right) \varphi_i  \right| \right\}.
	\end{split}
\end{align}
Using the local error for functions $\varphi_i \in \C^{\infty}_b(\RR^N;\RR)$ contained in \eqref{eq:Lagglobalsmooth}, we get
\begin{align}
	\begin{split}
		\left|\left[ Q^{y}_{t_j,t_{j+1}}- Q_{{t_j,t_{j+1}}}^{E^L,y} \right] u^y(t_{j+1},x) \right| 
		\leq   C (t_{j+1} - t_{j})^{1/2} \left\{ n^{-r} + r \sum_{i=0}^{j} \left\{ n^{-r}  + (t_{i+1}-t_i)^{(l+1)/2} \right\} \right\}.
	\end{split}
\end{align}
Now, we simply use that $t_{j+1}-t_j \leq C/n$ to get:
\begin{align}
	\begin{split}
		\left|\left[ Q^{y}_{t_j,t_{j+1}}- Q_{{t_j,t_{j+1}}}^{E^L,y} \right] u^y(t_{j+1},x) \right| 
		& \leq   C n^{-1/2} \left\{ n^{-r} + r n \left\{ n^{-r}  + n^{-(l+1)/2} (1-r/n)^{-l/2} \right\} \right\} \\
		& \leq C \left( n^{1/2-r} + n^{-l/2}(1-r/n)^{-l/2} \right)
	\end{split}
\end{align}
Combining with the local cubature errors and summing up, we get the global error:
\[
\left[ P^y_{0,t_{n}}- Q_{{t_{n}}}^{E^L,y,\Pi^{\gamma,r}_n} \right] f(x)  \leq C \left(	n^{3/2-r} + n^{-(l-2)/2}(1-r/n)^{-l/2} \right).
\]

\section{Numerical Examples}
\label{sec:NumEx}

\subsection{Example 1}

In this section, we implement and compare both algorithms.
%
%
%
We consider the following example with dimensions $N=d=1$:
\[
X_t^{0,x} = x + \int_0^t \EE \left[X_s^{0,x}\right] \, ds + B_t,
\]
which has the explicit solution
\[
X_t^x= x e^t + B_t.
\]
In this case, $ \EE X_t^{0,x} = x e^t$ so the Taylor approximation of order $q$ is easy to compute:
\[
	\mathcal{T}^q_t \left( \EE X_t^{0,x} \right) = \sum_{k=0}^q \frac{x}{k!} \, t^k.
\]
We choose the Lipschitz terminal function $f(x)= x^+:=\max\{x,0\}$ and, by integrating the Gaussian density, we can compute
\[
\EE (X_t^{0,x})^+ = \sqrt{t} \phi\left( \frac{x e^t}{\sqrt{t}} \right) + xe^t \left(1-\Phi\left(-\frac{x e^t}{\sqrt{t}}\right)\right),
\]
where $\phi$ and $\Phi$ are the density and cumulative distribution function, respectively, of a standard Gaussian random variable. We use the cubature formula of degree 5 contained in Lyons \& Victoir \cite{lyons2004cubature}. We use a fourth order adaptive Runge-Kutta scheme to solve the ODEs. We choose our parameters in order to achieve the optimal rate of convergence as given by Theorem \ref{th:MKVCub}. Since the coefficients are uniformly elliptic, expect to be able to achieve order 2 convergence with a cubature formula of degree 5. So, we only need to choose $q \geq 1 $ and $\gamma \in (4,5)$ to achieve quadratic convergence in the Taylor method, and $r \geq 3$ in the Lagrange interpolation method. We choose the parameters $(x,T,\gamma,q,r)=(0.5,10,4.5,2,3)$
and the results are presented in Figure \ref{fig:example1}. We fit a line to the last four points on the log-log error plot and calculate its gradient as an estimate of the rate of convergence.
\begin{figure}[h]
	\centering
	\includegraphics[width=0.8\textwidth]{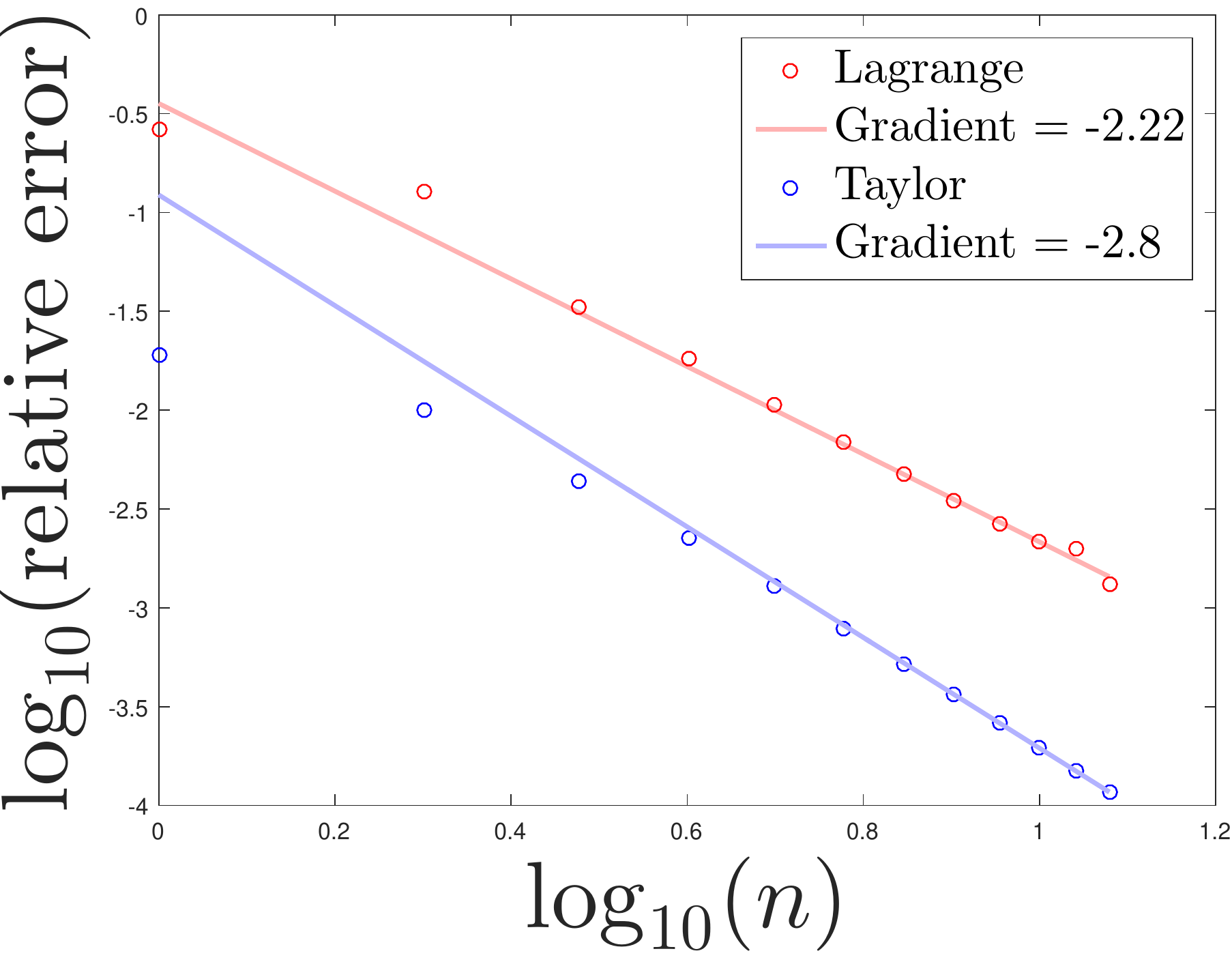}
	\captionsetup{width=0.8\textwidth}
	\caption{log-log error plot comparison between the Lagrange interpolation and Taylor methods for Example 1. The gradient of each solid line is given by a linear regression on the last 5 points.}
	\label{fig:example1}
\end{figure}
\noindent We see that both methods achieve the expected quadratic convergence rate. In this simple example, the convergence is quite smooth and the Taylor method performs better than the Lagrange interpolation method.

\subsection{Example 2}

We implement an example where the coefficients are not uniformly elliptic and $N=d=2$. We write $X_t^{0,x} = \left(X_t^{1} ,X_t^{2}\right)$ to lighten notation slightly. The example we consider is
\[
\begin{pmatrix}
	X_t^{1} \\ X_t^{2}
\end{pmatrix}
=
 \begin{pmatrix}
 x_1 \\ x_2
 \end{pmatrix}
  + \int_0^t  \begin{pmatrix} \left[2 + \sin \left(  \EE \, X_s^{2}\right) \right]
   \\  X_t^{1}
  \end{pmatrix} \circ dB^1_s
   + \int_0^t \begin{pmatrix}
   X_t^{2} \\ X_t^{1}
   \end{pmatrix} \circ dB^2_s,
\]
where the coefficients are 
\[
V_0 \equiv 0, \quad V_1(x_1,x_2,x') = \begin{pmatrix}
2 + \sin(x') \\ x_1
\end{pmatrix}, \quad 
V_2(x_1,x_2,x')= \begin{pmatrix}
x_2 \\ x_1
\end{pmatrix}, \quad \varphi_1(x_1,x_2)=x_2, 
\]
for all $(x_1,x_2,x')\in \RR ^3$.
We note that at $x_1=0$ the coefficients degenerate. Second, we note that
\[
V_{[(1,2)]}(x_1,x_2,x') = \begin{pmatrix}
x_1 \\ 2 + \sin(x') - x_1
\end{pmatrix}.
\]
Since $x_1$ and $2 + \sin(x') - x_1$ cannot both be zero at the same time, we see that $V_1, V_2$ and $V_{[(1,2)]}$ span $\RR^2$.
The coefficients therefore satisfy Assumption \ref{ass:MKVcub} (A1), the uniform strong H\"ormander condition, for $m=2$.
For $m=2$, with a cubature formula of degree 5, we expect to achieve a convergence rate of $3/2$ according to Theorem \ref{th:MKVCub}. To do so, we have to choose $\gamma\in (4,5)$ and $r >7/2$.
We choose the parameters $(x_1,x_2,T,\gamma,r)=(1,0.5,1,4.5,4)$, with the terminal function $f(x)=x^+$. We implement the cubature formula of degree 5 in dimension $d=2$ from Lyons \& Victoir \cite{lyons2004cubature}. In this case, the cubature measure is supported on $N_{Cub}=13$ paths. We could not find an explicit solution, so we compare the cubature approximation to a Monte Carlo approximation with Euler-Maruyama discretisation. The results are presented in Figure \ref{fig:example2}.
\begin{figure}[h]
	\centering
	\includegraphics[width=0.75\textwidth]{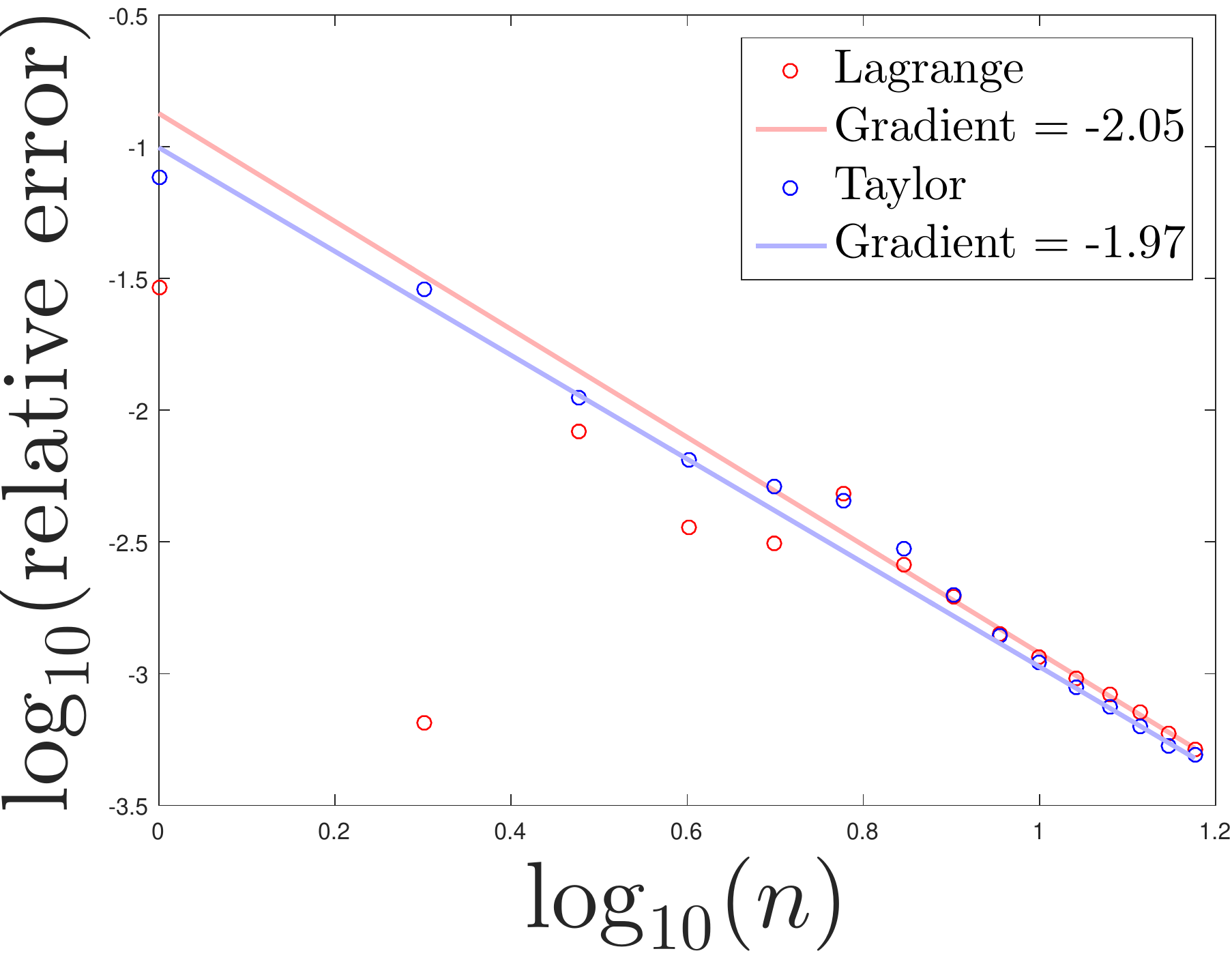}
	\captionsetup{width=0.8\textwidth}
	\caption{log-log error plot comparison between the Lagrange interpolation and Taylor methods for Example 2. The gradient of each solid line is given by a linear regression on the last 5 points.}
	\label{fig:example2}
\end{figure}
In this example, the convergence is not as smooth as Example 1, but the $\log$-$\log$ error plot looks approximately linear after 7 steps. After this, the performance of each algorithm is remarkably similar. Empirically we observe second order convergence, whereas Theorem \ref{th:MKVCub} predicts a rate of $3/2$.

\appendix

\section{Derivative estimates for time-inhomogeneous parabolic PDEs}
\label{sec:timeinhomintro}
In this section, we obtain estimates on the derivatives of the solution of the linear parabolic  partial differential equation (PDE)
\begin{align}
\label{eq:PDE}
\begin{split}
\left(\partial_t  + \mathcal{L}_t \right) u (t,x)=0 &, \:\:\:\:\:\: (t,x) \in [0,T) \times \RR^N  \\
u(T,x) =f(x) &, \:\:\:\:\:\: x \in \RR^N 
\end{split}
\end{align}
where $f$ is either Lipschitz or continuous and bounded, and $\mathcal{L}_t$ is the \emph{time-inhomogeneous} differential operator, written in H\"{o}rmander form,
\[
\mathcal{L}_t = W_0(t) + \frac{1}{2} \sum_{i=1}^d W_i(t)^2.
\]
The connection between parabolic PDEs and stochastic differential equations has been well-studied. Under various types of conditions on the vector fields $W_0, \ldots, W_d$ and terminal condition $f$, the solution to \eqref{eq:PDE} is given by $u(t,x)= \EE[f(X_T^{t,x})]$, where 
$(X^{t,x}_s)_{s \in [t,T]}$ solves the following Stratonovich SDE driven by a Brownian motion $B=(B^1, \ldots, B^d)$, with the convention $B^0_t=t$,
\begin{equation}
\label{eq:SDE}
X_s^{t,x} =  x + \sum_{i=0}^d \int_t^s W_i(u,X_u^{t,x}) \circ dB^i_u .
\end{equation}
In the homogeneous case, Kusuoka \& Stroock \cite{KusStrII} under a uniform H\"{o}rmander condition and subsequently Kusuoka \cite{Kusuoka3} under the weaker UFG condition, establish sharp estimates on the derivatives of the solution of \eqref{eq:PDE}. Crisan \& Delarue \cite{crisandelarue} extend this analysis to semi-linear equations. To our knowledge, these results have not been obtained without using Malliavin Calculus.

When the vector fields defining the time-inhomogeneous SDE are smooth in time and space, we can consider the space-time process $\left(t,X_t^{0,x}\right)_{t \in [0,T]}$ on $\RR^{N+1}$ and adapt existing results in the literature. There are three main works we draw on: we first show how to adapt the results of Kusuoka \cite{Kusuoka3} to derive gradient bounds in the directions of the vector fields except $W_0$; we then adapt an argument from Crisan \& Delarue \cite{crisandelarue} to prove that the time-inhomogeneous semigroup is a generalised classical solution to a parabolic PDE, and finally using this PDE as a tool, we adapt a result from Crisan, Manolarakis \& Nee \cite{crisan2010cubature} to derive gradient bounds in the direction of all the vector fields, including $W_0$.

We introduce the first assumption which we make throughout this section.
\begin{assumption}
	\label{ass:smoothbounded}
	$W_0, \ldots, W_d \in \C^{\infty}_b([0,T]\times \RR^N;\RR^N)$.
\end{assumption}
Under this assumption, \eqref{eq:SDE} has a unique strong solution. Now, let us define the space-time process $\left(\widetilde{X}^{(t,x)}_s\right)_{s \in [t,T]} :=\left(s, X^{t,x}_s\right)_{s \in[t,T]}$, taking values in $\RR^{N+1}$, which solves the equation
\begin{equation}
\label{eq:spacetime}
\widetilde{X}^{(t,x)}_s= \begin{pmatrix}t \\ x \end{pmatrix} + \int_t^s \begin{pmatrix}1 \\ W_0 \left(\widetilde{X}^{(t,x)}_u \right) \end{pmatrix} \, du + \sum_{i=1}^d \int_t^s \begin{pmatrix}0 \\ W_i\left(\widetilde{X}^{(t,x)}_u \right) \end{pmatrix} \circ dB^i_u .
\end{equation}
Defining $\widetilde{x}:=(t,x) \in \RR^{N+1}$ and $\widetilde{W}_0, \ldots, \widetilde{W}_d: \RR^{N+1} \to \RR^{N+1}$ as follows:
\begin{align*}
\forall \widetilde{x} \in \RR^{N+1}: \quad 	\widetilde{W}_0(\widetilde{x}):= \begin{pmatrix}1 \\ W_0(\widetilde{x}) \end{pmatrix}, \quad \widetilde{W}_i(\widetilde{x}):= \begin{pmatrix}0 \\ W_i(\widetilde{x}) \end{pmatrix} \: \text{ for } i=1, \ldots, d,
\end{align*}
we can re-write \eqref{eq:spacetime} more compactly as
\begin{equation}
\label{eq:spacetime2}
\widetilde{X}^{\widetilde{x}}_s= \widetilde{x} +  \sum_{i=0}^d \int_t^s \widetilde{W}_i\left(\widetilde{X}^{\widetilde{x}}_u \right) \circ dB^i_u .
\end{equation}
As we shall see, by working under the relaxed $\widetilde{UFG}$ condition (see Assumption \ref{ass:UFG}), the solution of the PDE \eqref{eq:PDE} is not necessarily differentiable in each co-ordinate direction in $\RR^N$. The solution of \eqref{eq:PDE} remains differentiable in certain directions, determined by the vector fields $\widetilde{W}_0, \ldots, \widetilde{W}_d$. We now explain what we mean by such a directional derivative.

We can identify vector fields, $U:\RR^{m} \to \RR^{m}$ with differential operators acting on sufficiently smooth functions $\varphi :\RR^{m} \to \RR$ by
\begin{equation}
\label{eq:dirderivsmth}
\forall y \in \RR^m: \quad	U \varphi (y) :=  \nabla \varphi(y) \, U(y) = \sum_{i=1}^{m} U^i(y) \, \partial_{y_i}  \varphi(y)
\end{equation}
We can define a directional derivative of $\varphi$ in the direction $U$, even when $\partial_{y_j} \varphi$ does not exist classically for all $j=1, \ldots, m$. Let $w_s(y)$ be the solution to the ODE
\begin{align}
\label{eq:dirODE}
\begin{split}
&	\frac{dw_s(y)}{ds} = U(w_s(y)), \quad s \geq 0\\
&  w_0(y)=y.
\end{split}
\end{align}
We say that $\varphi$ is differentiable in the direction $U$ if the function $s \mapsto \varphi(w_s(y))$ is differentiable at 0. Then, we denote
\[
U \varphi(y)= \left. \frac{d}{ds} \varphi ( w_s(y)) \right|_{s=0},
\]
which coincides with \eqref{eq:dirderivsmth} when $\varphi \in \C^1(\RR^m;\RR)$. In fact, we will see that the semigroup associated to equation is differentiable in directions determined by commutators of the vector fields. The Lie bracket, or commutator, between two vector fields $U$ and $W$ is then defined the differential operator
\[
[U,W]\varphi := U (W \varphi) - W (U \varphi),
\]
which can be identified with the vector field
\[
[U,W](y) = \partial W (y) U(y) - \partial U(y) W(y),
\]
where $\partial W(y):= (\partial_{y_j} W^i(y))_{1 \leq i,j \leq m}$ is the Jacobian matrix of $W$ and similarly for $\partial U$.

Since we have assumed the vector fields $\widetilde{W}_0, \ldots, \widetilde{W}_d$ to be smooth, we can repeatedly take commutators of them.
Recall the notation $\A$ for multi-indices on $\{0, \ldots, d\}$ from section \ref{sec:introCub}. We define $\widetilde{W}_{[\alpha]}$, for $\alpha \in \A$ inductively by forming Lie brackets on $\RR^{N+1}$:
\begin{align*}
\widetilde{W}_{[i]} := \widetilde{W}_i , \:\:\:\:\: \widetilde{W}_{[\alpha*i]} := [\widetilde{W}_{[\alpha]}, W_i]  \:\:\: \text{ for } \: i=0 \ldots, d, \:\: \alpha \in \A.
\end{align*}
We note that for all $\alpha \in \A_1(m)$ (i.e. for $\alpha\neq(0)$) the first component of the vector field $\widetilde{W}_{[\alpha]}$ is zero.
So for $\alpha \in \A_1(m)$, a derivative in the direction  $\widetilde{W}_{[\alpha]}$ of a function $\RR^{N+1} \ni (t,x) \mapsto \phi (t,x) \in \RR$ only acts in the $x$ variable. We can therefore write $\{\widetilde{W}_{[\alpha]}(t): \alpha \in \A_1(m)\}$ and think of these as differential operators parametrised by $t$ and acting in the $x$ variable. Only the vector field $\\widetilde{W}_0$ acts in the $t$-direction.  

With these concepts in mind, we can now introduce the second assumption we make on the vector fields.
\begin{assumption}
	\label{ass:UFG}
	$\widetilde{UFG}$($m$) condition: there exists a positive integer $m$ such that, for all $\alpha \in \A$ with $\|\alpha\|>m$, there exist  $\varphi_{\alpha, \beta} \in \C_b^{\infty}([0,T]\times \RR^N;\RR)$ with
	\[
	\widetilde{W}_{[\alpha]}(t,x) = \sum_{\beta \in \A_1(m)} \varphi_{\alpha, \beta}(t,x) \, \widetilde{W}_{[\beta]}(t,x).
	\]
\end{assumption}

\subsection{Kusuoka-Stroock processes}
\label{sec:introKSP}

This class of process was introduced by Kusuoka \& Stroock \cite{KusStrIII}.
They will appear as Malliavin weights in our integration by parts formulas. The definition and properties we give here record the regularity and growth of these processes with respect to different parameters. The results allow one to develop integration by parts formulas in a systematic and transparent way, which automatically leads to nice derivative estimates.

\begin{definition}[Kusuoka-Stroock processes] \label{def:localKS}
	Let $E$ be a separable Hilbert space and let $r \in \RR$, $M \in \mathbb{N}$. We
	denote by $\K_{r}(t,E,M)$ the set of functions: $\Phi_t : (t,T] \times
	\RR^N  \to \mathbb{D}^{M,\infty}(E)$ satisfying the
	following:
	\begin{enumerate}
		\item For all $s \in (t,T]$,  the map $\RR^N \ni x \mapsto \Phi_t(s,x) \in L^p(\Omega)$ is $M$-times continuously differentiable for all $p \geq 1$.
		\item For any $p \geq 1$, any multi-index $\alpha$ on $\{1, \ldots, N\}$ and $m \in \NN$ with $|\alpha| +m \leq M$, we have 
		\begin{align}\label{ksineqinx}
		\begin{split}
		\sup_{x \in \RR^N}
		\sup_{s \in (t,T]} (s-t)^{-r/2} & \left\| \partial^{\alpha}_x \Phi_t(s,x) \right\|_{ \DD^{m,p}(E) }   <\infty.
		\end{split}
		\end{align}
	\end{enumerate}
\end{definition}
\begin{remark}
	\begin{enumerate}
		\item The number $M$ denotes how many times the Kusuoka-Stroock function can be differentiated and $r$ measures the growth in $(s-t)$.
		\item This definition is slightly different to that in \cite{KusStrIII}: here our processes are defined on $(t,T]$ instead of $(0,T]$ and we require continuity in $L^p(\Omega)$ rather than almost surely.
	\end{enumerate}
\end{remark}

We record here some properties which help when building Malliavin weights later.
\begin{lemma}
	\label{lem:KSprops}
	\begin{enumerate}
		\item[a.] If $\Phi_t \in \K^q_r(t,\RR,M)$ is $\mathbb{F}$-adapted, and we define \newline $\textstyle \Psi^i_t(s,x):= \int_t^s \Phi_t(u,x) \, dB^i_u$ for $i \in \{1, \ldots, d\}$ and $\textstyle \Gamma(s,x):= \int_t^s \Phi_t(u,x) \, du$, then $\Psi_t^i \in \K^q_{r+1}(t,\RR,M)$ and $\Gamma_t \in \K^q_{r+2}(t,\RR,M)$.
		\item[b.] If $\Phi_{t,i} \in \K^{q_i}_{r_i}(t,\RR,M_i), i=1, \ldots, n$ then $ \textstyle \sum_{i=1}^n \Phi_{t,i} \in \K^{\max_i q_i}_{\min_i r_i}(t,\RR, \min_i M_i)$ and \newline $ \textstyle\prod_{i=1}^n \Phi_{t,i} \in \K^{q_1 + \ldots + q_n}_{r_1 + \ldots + r_N}(t,\RR, \min_i M_i )$.
	\end{enumerate}
\end{lemma}
\begin{proof}
	The proof is essentially the same as the proof of Lemma 75 in \cite{crisan2010cubature}.
\end{proof}

\subsection{Integration by parts \& derivative bounds}
\label{sec:GBs}

We are now in a position to prove the main integration by parts result.

\begin{theorem}
	\label{th:IBPandGB}
	Assume that $\widetilde{UFG}(m)$ holds and fix $s \in (t,T]$. Then, for any $\alpha_1, \ldots, \alpha_n \in \A_1(m)$, there exist $\Phi^1_{t,\alpha_1, \ldots, \alpha_n} \in \K_0(t,\RR)$ and $\Phi^2_{t,\alpha_1, \ldots, \alpha_n} \in \K_0(t,\RR^N)$ such that for $f \in \C^{\infty}_b(\RR^N;\RR)$,
	\begin{equation}
	\label{eq:bdedIBP}
	\widetilde{W}_{[\alpha_1]}(t) \cdots \widetilde{W}_{[\alpha_n]}(t) \left( P_{t,s}f \right)(x) = (s-t)^{- \frac{\| \alpha_1 \|+ \ldots, + \| \alpha_{n} \|}{2}} \EE \left[ f(X_s^{t,x}) \,  \Phi^1_{t,\alpha_1, \ldots, \alpha_n}(s,x) \right],
	\end{equation}
	and
	\begin{equation}
	\label{eq:LipIBP}
	\widetilde{W}_{[\alpha_1]}(t) \cdots \widetilde{W}_{[\alpha_n]}(t) \left( P_{t,s}f \right)(x) = (s-t)^{- \frac{\| \alpha_1 \|+ \ldots, + \| \alpha_{n-1} \|}{2}} \EE \left[ \nabla f(X_s^{t,x}) \,  \Phi^2_{t,\alpha_1, \ldots, \alpha_n}(s,x) \right].
	\end{equation}
	Moreover, for $f$ continuous and bounded or Lipschitz, $ P_{t,s}f(x)$ is differentiable in the directions $\{\widetilde{W}_{[\alpha]}(t): \alpha \in \A_1(m)\}$ with 
	\begin{equation}
	\label{eq:GB1}
	\sup_{x \in \RR^N} \left|	\widetilde{W}_{[\alpha_1]}(t) \cdots \widetilde{W}_{[\alpha_n]}(t) \left( P_{t,s}f \right)(x) \right | \leq C \, \|f \|_{\infty} \, (s-t)^{\frac{-(| \alpha_1 |+ \ldots, + | \alpha_n |)}{2}} ,
	\end{equation}
	and
	\begin{equation}
	\label{eq:GB2}
	\sup_{x \in \RR^N} \left|	\widetilde{W}_{[\alpha_1]}(t) \cdots \widetilde{W}_{[\alpha_n]}(t) \left( P_{t,s}f \right)(x) \right | \leq C \, \|f \|_{\text{Lip}} \, (s-t)^{\frac{1-(| \alpha_1 |+ \ldots, + | \alpha_n |)}{2}} ,
	\end{equation}
	in each case, respectively.
\end{theorem}

\begin{remark}
	We emphasise here that each $\widetilde{W}_{[\alpha_1]}(t)$ is a differential operator acting in $x$, parametrised by $t$, applied to the function $x \mapsto P_{t,s}f(x)$. These results are not valid for e.g. $\widetilde{W}_{[\alpha_1]}(u)P_{t,s}f(x)$ when $u \neq t$. We develop estimates on the derivative of $(t,x) \mapsto P_{t,s}f(x)$ as a function on $[0,T] \times \RR^N$ in Proposition \ref{prop:derivs}.
\end{remark}

\begin{proof}[Proof of Theorem \ref{th:IBPandGB}]
	Since the $\widetilde{UFG}$ condition is precisely the UFG condition of Kusuoka on $\RR^{N+1}$, we can use his results. We use the notation $\widetilde{P}_{s}g (\widetilde{x} ):= \EE\left[ g \left( \widetilde{X}^{\widetilde{x}}_s\right) \right]$ for a suitably integrable function $g : \RR^{N+1} \to \RR$.
	By Kusuoka \cite{Kusuoka3} Lemma 8 (see also \cite{crisan2010cubature} Corollary 32), we know for any $\alpha_1, \ldots, \alpha_n \in \A_1(m)$, there exist $\Phi^1_{\alpha_1, \ldots, \alpha_n} \in \K_0(\RR)$ and $\Phi^2_{\alpha_1, \ldots, \alpha_n} \in \K_0(\RR^{N+1})$ such that for $g \in \C^{\infty}_b(\RR^{N+1};\RR)$,
	\begin{equation*}
	\widetilde{W}_{[\alpha_1]}(t) \cdots \widetilde{W}_{[\alpha_n]}(t) \left( \widetilde{P}_{s}g \right)(\widetilde{x} ) = (s-t)^{- \frac{\| \alpha_1 \|+ \ldots, + \| \alpha_{n} \|}{2}} \EE \left[ g( \widetilde{X}_s^{\widetilde{x}}) \,  \Phi^1_{\alpha_1, \ldots, \alpha_n}(s,\widetilde{x}) \right],
	\end{equation*}
	and
	\begin{equation*}
	\widetilde{W}_{[\alpha_1]}(t) \cdots \widetilde{W}_{[\alpha_n]}(t) \left( \widetilde{P}_{s}g \right)(\widetilde{x}) = (s-t)^{- \frac{\| \alpha_1 \|+ \ldots, + \| \alpha_{n-1} \|}{2}} \EE \left[ \nabla g( \widetilde{X}_s^{\widetilde{x}}) \,  \Phi^2_{\alpha_1, \ldots, \alpha_n}(s,\widetilde{x}) \right].
	\end{equation*}
	Now for any function $f \in \C^{\infty}_b(\RR^N;\RR)$, we can extend it to $g \in \C^{\infty}_b(\RR^{N+1};\RR)$ by $g(t,x):=f(x)$. We then immediately have the integration by parts formulas \eqref{eq:bdedIBP} and \eqref{eq:LipIBP}.
	We get the bound stated in \eqref{eq:GB1} with the constant
	\[
	C= \sup_{s \in [t,T]} \sup_{x \in \RR^N} \EE \left| \Phi^1_{t, \alpha_1,\ldots, \alpha_n }(s,x)\right|,
	\]
	and a standard approximation argument gives the same estimate for bounded and continuous $f$. Similarly, we obtain the bound in \eqref{eq:GB2} with constant
	\[
	C= \sup_{s \in [t,T]} \sup_{x \in \RR^N} \EE \left| \Phi^2_{t, \alpha_1,\ldots, \alpha_n }(s,x)\right|,
	\]
	and a standard approximation argument allows one to obtain the same bound for $f$ Lipschitz with $\| \nabla f\|_{\infty}$ replaced by $\|f \|_{Lip}$.
\end{proof}

\subsection{Uniform H\"ormander setting}
\label{sec:hormander}
In this section, let us consider a stronger assumption than the $\widetilde{UFG}(m)$ condition. Suppose that a uniform strong H\"ormander condition of order $m$ holds - that is
\begin{assumption}
	\label{ass:USH}
	USH($m$) : There exists $\delta>0$ and $m \in \NN$ such that for all $\xi \in \RR^N$,
	\[
	\inf_{(t,z) \in [0,T] \times \RR^N} \sum_{\alpha \in \A_{\geq 1}(m)} \langle W_{[\alpha]}(t,z) , \xi \rangle^2 \geq \delta \, | \xi |^2.
	\]
\end{assumption}
In this case, we recover differentiability of $x \mapsto P_{t,T}f(x)$ in all directions.

\begin{corollary}
	\label{cor:hormanderIBP}
	Assume USH($m$) holds. Let $\eta$ be a multi-index on $\{1, \ldots, N\}$ and let $f$ be Lipschitz. Then,
	\[
	\sup_{x \in \RR^N} \left| \partial^{\eta} \left(P_{t,T} f\right)(x) \right| \leq  C \, \|f\|_{Lip}  \,(T-t)^{-(|\eta|-1)m/2}.
	\]
\end{corollary}
\begin{proof}
	First let $f \in \C^{\infty}_b(\RR^N;\RR)$.
	For the first order derivatives, 
	\begin{align*}
	\partial^i_{x} \EE \left[ f(X^{t,x}_T) \right] = \sum_{k=1}^N \EE \left[ \partial^k f (X^{t,x}_T) \, \partial^i_x (X^{t,x}_T)^k\right].
	\end{align*}
	For the higher order derivatives, we note that there exist $F^i_{\alpha} \in \C^{\infty}_b([0,T] \times \RR^N;\RR)$ such that
	\[
	e_j = \sum_{\alpha \in \A_{\geq 1}(m)} F^j_{\alpha}(t,x) \, W_{[\alpha]}(t,x),
	\]
	where $e_j$ is the $j$-th standard basis vector in $\RR^N$. To see this, define $\W(t,x)$ to be the $N \times \text{card}(\A_{\geq 1}(m))$ matrix whose columns are the vector fields $(W_{[\alpha]})_{\alpha \in  \A_{\geq 1}(m)}$ evaluated at $(t,x)$. USH($m$) guarantees that $\W \W^{\top}(t,x)$ is invertible. Then, 
	\[
	F^j_{\alpha}(t,x) : = \left( \W^{\top} [\W \W^{\top}]^{-1}(t,x) \, e_j \right)_{\alpha}
	\]
	satisfies the above relation. Then, for the second order derivatives,
	\begin{align*}
	\partial^{(j,i)}_{x} \EE \left[   f(  X_T^{t,x}) \right]   &= 	\sum_{k=1}^N \partial^j_x \EE \left[ \partial^k f (X^{t,x}_T) \, \partial^i_x (X^{t,x}_T)^k\right]  \\
	&= \sum_{k=1}^N \sum_{\alpha \in \A_{\geq 1}(m)} F^j_{\alpha}(t,x) \, W_{[\alpha]}(t) \left( \EE \left[ \partial^k f (X^{t,x}_T) \, \partial^i_x (X^{t,x}_T)^k\right] \right).
	\end{align*}
	We note that $\partial^i_x (X^{t,x}_T)^k \in \K_0(t,\RR)$, so we can apply the IBPF in Kusuoka \cite{Kusuoka3} Lemma 8 to obtain the existence of $H_{t,\alpha} \in \K_0(t,\RR)$ such that
	\begin{align*}
	\partial^{(j,i)}_{x} \EE \left[   f(  X_T^{t,x}) \right]  
	&= \sum_{k=1}^N \sum_{\alpha \in \A_{\geq 1}(m)} F^j_{\alpha}(t,x) \, (T-t)^{-|\alpha|/2} \, \EE \left[ \partial^k f (X^{t,x}_T) \,H_{t,\alpha} \left( T,x \right)\right] \\
	&= (T-t)^{-m/2} \sum_{k=1}^N  \EE \left[ \partial^k f (X^{t,x}_T) \,\bar{H}_{t,i,j,k}(T,x) \right],
	\end{align*}
	where
	\[
	\bar{H}_{t,i,j,k} (T,x):= \sum_{\alpha \in \A_{\geq 1}(m)}  (T-t)^{(m-|\alpha|)/2} F^j_{\alpha}(t,x) H_{t,\alpha} \left( T,x \right) \in \K_0(t,\RR),
	\]
	and we have used that for all $\alpha \in \A_{\geq 1}(m)$, $\| \alpha\|=|\alpha| \leq m$. We get the bound:
	\begin{align*}
	\sup_{x \in \RR^N}	\left|\partial^{(j,i)}_{x} \EE \left[   f(  X_T^{t,x}) \right]  \right|
	\leq (T-t)^{-m/2} \| \nabla f\|_{\infty} \, \sum_{k=1}^N \sup_{x \in \RR^N}   \EE \left|\bar{H}_{t,i,j,k}(T,x) \right|.
	\end{align*}
	We can iterate this argument as many times as we like. Then we can get the bound for a Lipschitz $f$ using approximation as before.
\end{proof}

\subsection{Connection with PDE}
\label{sec:PDE}

In this section, we make use of the integration by parts formulae of Theorem \ref{th:IBPandGB} to extend the notion of classical solution to the PDE \eqref{eq:PDE} to the case when the solution is not classically differentiable in all directions. The notation and arguments in this section closely follows Crisan \& Delarue \cite{crisandelarue}, who provide a similar notion of solution to semilinear PDEs with coefficients which do not depend on time. The idea is very simple: it is a standard result that for a terminal condition $f \in \C^{\infty}_p(\RR^N;\RR)$ the PDE \eqref{eq:PDE} has a classical solution. For $f \in \C_p(\RR^N;\RR)$, we consider a sequence of smooth approximations $(f_l)_{l \geq 1}$ to which we can associate solutions $(v_l)_{l \geq 1}$ to \eqref{eq:PDE}. For each $v_l$ we can use the integration by parts formula of Theorem \ref{th:IBPandGB} to write the derivatives $W_i^2(t)v_l(t,x)$ in a form which does not depend on any derivatives of $f_l$. We then show that the PDE still holds in the limit $l \to \infty$.

We introduce some function spaces we will need to define what we mean by a classical solution. We denote by $\BB(0,R)$  the open ball in $\RR^N$ of radius $R>0$ centred at zero. Let $\phi \in \C^{\infty}_b([0,T-1/r]\times \BB(0,R);\RR)$ and define
\[
\| \phi \|^{\widetilde{W}_0, 1}_{[0,T-1/R]\times \BB(0,R);\infty } := \| \phi \|_{[0,T-1/R]\times \BB(0,R);\infty} + \| \widetilde{W}_0(t) \phi \|_{[0,T-1/R]\times \BB(0,R);\infty} .
\]
Define $\calD^{1, \infty}_{\widetilde{W}_0}([0,T-1/R]\times \BB(0,R))$ as the closure of $\C^{\infty}_b([0,T-1/R]\times \BB(0,R);\RR)$ in $\C_b([0,T-1/R] \times \bar{\BB}(0,R);\RR)$ with respect to the norm $\| \cdot \|^{\widetilde{W}_0, 1}_{[0,T-1/R]\times \BB(0,R);\infty }$. And define
\[
\calD^{1, \infty}_{\widetilde{W}_0}([0,T)\times \RR^N):= \bigcap_{R \geq 1} \calD^{1, \infty}_{\widetilde{W}_0}([0,T-1/R]\times \BB(0,R)).
\]
Now, take $\psi \in \C^{\infty}_b(\RR^N;\RR)$ and, for any ball $\BB$, define
\[
\| \psi \|^{W(t),1}_{\BB, \infty} := \|\psi\|_{\BB, \infty} + \sum_{i=1}^d \| W_i(t) \psi \|_{\BB, \infty},
\]
and
\[
\| \psi \|^{W(t),2}_{\BB, \infty} :=	\| \psi \|^{W(t),1}_{\BB, \infty} + \sum_{i=1}^d \| W_i^2(t) \psi \|_{\BB, \infty}.
\]
We define $\calD_{W(t)}^{2, \infty} (\BB)$ to be the closure of $\C^{\infty}_b(\BB;\RR)$ in $\C_b(\bar{\BB};\RR)$ with respect to $	\| \cdot \|^{W(t),2}_{\BB, \infty}$ and 
\[
\calD_{W(t)}^{2, \infty} (\RR^N) := \bigcap_{	R \geq 1} \calD_{W(t)}^{2, \infty} (\BB(0,R)).
\]

\begin{definition}[Classical solution]
	\label{def:classical}
	We define a function $v:[0,T]\times \RR^N \to \RR$ to be a classical solution to \eqref{eq:PDE} if the following three conditions are satisfied:
	\begin{enumerate}
		\item $v \in \calD^{1, \infty}_{\widetilde{W}_0}([0,T)\times \RR^N)$ and for each $t \in [0,T)$, $v(t, \cdot) \in 	\calD_{W(t)}^{2, \infty} (\RR^N)$, such that for $i=1, \ldots, d$,
		\[
		[0,T) \times \RR^N \ni(t,x) \mapsto \left(W_i(t)v(t,x), W_i^2(t) v(t,x) \right)
		\]
		is a continuous function.
		\item For all $(t,x) \in [0,T) \times \RR^N$,
		\[
		\widetilde{W}_0v(t,x) + \frac{1}{2} \sum_{i=1}^d W_i^2 v(t,x) = 0
		\]
		\item $\lim_{(t,y) \to (T,x)} v(t,y) = f(x)$ for all $x \in \RR^N$.
	\end{enumerate}
\end{definition}

\begin{remark}
	\begin{enumerate}
		\item Note that since, in general, the space $\calD_{W(t)}^{2, \infty} (\RR^N)$ is different for each $t \in [0,T)$, our definition requires that $v(t, \cdot)$ belongs to a different space at each time $t \in [0,T)$.
		\item If $\phi \in \C^{1,2}([0,T)\times \RR^N;\RR)$, then $\phi \in \calD^{1, \infty}_{\widetilde{W}_0}([0,T)\times \RR^N)$ and $\phi(t,\cdot) \in \calD_{W(t)}^{2, \infty} (\RR^N)$ for all $t \in [0,T)$. Moreover,
		\begin{align*}
		\| \phi \|^{\widetilde{W}_0, 1}_{[0,T-1/R]\times \BB(0,r);\infty } & \leq C_R \, \big\{ \| \phi \|_{[0,T-1/r]\times \BB(0,r);\infty} 
		+ \| \partial_t \phi \|_{[0,T-1/R]\times \BB(0,R);\infty}  \\
		&  \quad \quad \quad +  \| \nabla \phi \|_{[0,T-1/R]\times \BB(0,R);\infty} \big\} \\
		\| \phi \|^{W(t),2}_{\BB(0,R), \infty} & \leq C_R \left\{   \|\phi(t, \cdot) \|_{\BB(0,R);\infty} +  \|\nabla \phi(t, \cdot) \|_{\BB(0,R);\infty} + \| \nabla^2 \phi(t, \cdot) \|_{\BB(0,R);\infty} \right\},
		\end{align*}
		where
		\begin{align*}
		C_R = 1 + \|W_0 \|_{[0,T-1/R]\times \BB(0,R);\infty} + \sum_{i=1}^d  \|W_i(t,\cdot)\|_{\BB(0,R);\infty} + \|\partial W_i(t,\cdot)\|_{\BB(0,R);\infty}.
		\end{align*}
		It is then clear that our definition truly is an \emph{extension} of the usual definition of classical solution.
	\end{enumerate}
\end{remark}

\noindent With this definition in hand, we have the following theorem.
\begin{theorem}
	\label{th:PDE}
	Assume that $\widetilde{UFG}(m)$ holds and let $f:\RR^N \to \RR$ be continuous with polynomial growth. Then, $v(t,x):=P_{t,T}f(x)$ is a classical solution to \eqref{eq:PDE}. It is also the unique solution amongst those which satisfy the following polynomial growth condition: there exists $q>0$ such that
	\[
	\left| v(t,x)\right| \leq C (1+|x|)^q \quad \forall t \in [0,T], x \in \RR^N
	\]
\end{theorem}

The proof of uniqueness relies on an It\^{o} formula valid for functions differentiable in the directions of the vector fields. We will also need a stochastic Taylor expansion based on this formula in Section \ref{sec:cuberror} for the analysis of the error in the cubature on Wiener space algorithm.

\begin{lemma}
	\label{lem:Itoformula}
	Let $v:[0,T) \times \RR^N \to \RR$ satisfy part (1) of Definition \ref{def:classical} and be of at most polynomial growth. Then, for all $u \in [t,T)$,
	\begin{align}
	\label{eq:Ito}
	\begin{split}
	v(u,X^{t,x}_u) = v(t,x) +  \int_t^u & \left[ \widetilde{W}_0(s) v(s,X_s^{t,x}) + \frac{1}{2} \sum_{i=1}^d W_i^2 (s)v(s,X_s^{t,x}) \right]  ds \\
	& + \sum_{i=1}^d \int_t^u W_i(s) v(s, X_s^{t,x}) d B^i_s.
	\end{split}
	\end{align}
\end{lemma}
\begin{proof}
	This can be proved by a mollification argument as in Proposition 7.1 in \cite{crisandelarue}.
\end{proof}

We can now prove Theorem \ref{th:PDE}.

\begin{proof}[Proof of Theorem \ref {th:PDE}]
	\emph{Existence}:
	Denote by $(f_l)_{l  \geq 1}$ a sequence of mollifications of $f$. Since $f$ is continuous, $f_l$ converges to $f$ uniformly on compact subsets of $\RR^N$. Since $v_l(t,x) - v(t,x)= \EE \left[ f_l(X^{t,x}_T) - f(X^{t,x}_T) \right]$, it is clear that $v_l$ converges to $v$ uniformly on compact subsets of $[0,T] \times \RR^N$. Therefore, $v$ is continuous up to the boundary at $t=T$. Now, consider the integration by parts formula for $W_i v_l$ and $W_i^2 v_l$ provided by \eqref{eq:bdedIBP} as part of Theorem \ref{th:IBPandGB}. We get
	\begin{align*}
	& W_i v_l(t,x) = (T-t)^{-1/2} \, \EE \left[ f_l(X_T^{t,x}) \Phi^1_{t,(i)}(T,x) \right], \\
	&  W_i^2 v_l(t,x) = (T-t)^{-1} \, \EE \left[ f_l(X_T^{t,x}) \Phi^1_{t,(i,i)}(T,x) \right],
	\end{align*}
	where, crucially, $\Phi^1_{t,(i)}, \Phi^1_{t,(i,i)}$ are independent of $f_l$. Then, considering the differences $W_i v_l(t,x)- W_i v_m(t,x)$ and $W_i^2 v_l(t,x)- W_i^2 v_m(t,x)$ over compact subsets of $[0,T) \times \RR^N$, we see that $(W_i v_l, W_i^2 v_l )_{l \geq 1}$ converges uniformly on compact subsets of $[0,T) \times \RR^N$. This proves that $W_iv, W_i^2v$ exist and are continuous. Now, each $f_l \in \C_p^{\infty}(\RR^N;\RR)$, so associated to each, there is a classical solution $v_l$ of the PDE \eqref{eq:PDE}. Since $\widetilde{W}_0 v_l = \textstyle - \frac{1}{2} \sum_{i=1}^d W_i^2 v_l$, and $W_i^2v_l \to W_i^2 v$ uniformly on compacts in $[0,T)\times \RR^N$, we get that $v \in  \calD^{1, \infty}_{\widetilde{W}_0}([0,T)\times \RR^N)$. Moreover, taking the limit in the PDE satisfied by $v_l$ shows that it is also satisfied by $v$.
	
	\emph{Uniqueness:} Using the It\^{o} formula in Lemma \ref{lem:Itoformula}, we have for $u<T$
	\begin{align*}
	v(u,X^{t,x}_u) = v(t,x) &+  \int_t^u \left[ \widetilde{W}_0(s) v(s,X_s^{t,x}) + \frac{1}{2} \sum_{i=1}^d W_i^2(s) v(s,X_s^{t,x}) \right]  ds \\
	& + \sum_{i=1}^d \int_t^u W_i(s) v(s, X_s^{t,x}) d B^i_s.
	\end{align*}
	Using part (2) of the definition, the drift term is zero and
	\begin{equation}
	\label{eq:PDEdiff}
	v(u,X^{t,x}_u) = v(t,x) +  \sum_{i=1}^d \int_t^u W_i(s) v(s, X_s^{t,x}) d B^i_s.
	\end{equation}
	
	Now, using that $v$ has polynomial growth and $X_u^{t,x}$ has moments of all orders, we can easily show that the left hand side of \eqref{eq:PDEdiff} is square integrable, and so the right hand side is too. Hence the right hand side is a true martingale
	and we can take expectation in \eqref{eq:PDEdiff} to get
	\[
	\EE v(u,X^{t,x}_{u}) = v(t,x)
	\]
	and using part (3) of the definition (continuity of $v$ at the boundary $t=T$) we can take $u \nearrow T$ to get
	\[
	\EE f(X^{t,x}_{T}) = v(t,x),
	\]
	which proves uniqueness.
\end{proof}

\subsection{Derivatives in the direction $\widetilde{W}_0$}
\label{sec:v0}

In Theorem \ref{th:IBPandGB}, we established integration by parts formulae for derivatives of $x \mapsto P_{t,T}f(x)$ in the directions $\{\widetilde{W}_{[\alpha]}(t), \alpha \in \A_1(m)\}$. However $0 \notin \A_1(m)$, so we have no control over derivatives in the direction $\widetilde{W}_0$. Using that $P_{t,T}f(x)$ solves PDE \eqref{eq:PDE} we are now able to estimate derivatives in the $\widetilde{W}_0$ direction.

\begin{proposition}
	\label{prop:derivs}
	Assume $\widetilde{UFG}(m)$ holds. Let $\alpha=(\alpha_1, \ldots, \alpha_n) \in \A$ and use the notation $\widetilde{W}_{\alpha}(t)= \widetilde{W}_{\alpha_1}(t) \cdots \widetilde{W}_{\alpha_n}(t)$. Then, the function $v(t,x):=P_{t,T}f(x)$ is differentiable in the directions $\widetilde{W}_0(t), W_1(t), \ldots, W_d(t)$ and the following bounds hold for all $t \in [0,T)$:
	for $f$ continuous and bounded,
	\begin{equation}
	\label{eq:PDEbdCts}
	\sup_{x \in \RR^N} \left|	\widetilde{W}_{\alpha}(t) v(t,x) \right | \leq C \, \|f \|_{\infty} \, (T-t)^{\frac{-(\| \alpha_1 \|+ \ldots, + \| \alpha_n \|)}{2}} .
	\end{equation}
	For $f$ Lipschitz,
	\begin{equation}
	\label{eq:PDEbdLip}
	\sup_{x \in \RR^N} \left|	\widetilde{W}_{\alpha}(t) v(t,x) \right | \leq C \, \|f \|_{\text{Lip}} \, (T-t)^{\frac{1-(\| \alpha_1 \|+ \ldots, + \| \alpha_n \|)}{2}}.
	\end{equation}
\end{proposition}

\begin{proof}
	Thinking of the $\widetilde{W}_0, W_1, \ldots, W_d$ as differential operators acting on functions in \newline
	$\C^{\infty}([0,T]\times \RR^N;\RR)$, Corollary 78 in \cite{crisan2010cubature} shows that 
	$\widetilde W_{\alpha}$, $\alpha \in \A$ satisfies  the following convenient identity 
	\begin{equation}\label{thebestindentity}
	\widetilde{W}_{\alpha}v=\sum_{i=1}^{\|\alpha\|}\sum_{\substack{\beta_{1}, \ldots,\beta_i\in \A_1,\\
			\|\beta_{1}\|+\ldots + \|\beta_{i}\| =\|\alpha\|}}
	c_{\alpha,\beta_{1}, \ldots,\beta_i} \widetilde{W}_{[\beta_{1}]}... \widetilde{W}_{[\beta_{i}]}v,
	\end{equation}
	where $c_{\alpha,\beta_{1}, \ldots,\beta_i}\in \mathbb{R}$. The importance of this identity is that the left hand side contains derivatives possibly in the direction $\widetilde{W}_0$ whereas on the right hand side, there are only derivatives in directions $\widetilde{W}_{[\alpha]}$, $\alpha \in \A_1$ which does not include $\widetilde{W}_0$.
	
	Hence,
	\[
	\left|\widetilde{W}_{\alpha} v \right|  \leq C \, \sup_{\substack{\beta_{1}, \ldots,\beta_i\in \A_1,\\
			\|\beta_{1}\|+\ldots + \|\beta_{i}\| =\|\alpha\|}}
	\left| \widetilde{W}_{[\beta_{1}]}... \widetilde{W}_{[\beta_{i}]}v \right|,
	\]
	this being exactly the type of term we can control by Theorem \ref{th:IBPandGB}. 
\end{proof}

\noindent Now, define, for $\varphi \in \C^{\infty}_b([0,T-1/R] \times \BB(0,R);\RR)$, the norm
\[
\| \varphi\|^{W,n}_{[0,T-1/R]\times\BB(0,R);\infty} : = \sum_{\alpha \in \A(n) } \| \widetilde{W}_{\alpha}  \varphi \|_{[0,T-1/R]\times\BB(0,R);\infty},
\]
and define $\widehat{\calD}^n([0,T-1/R] \times \BB(0,R))$ as the closure of $\C^{\infty}_b([0,T-1/R]\times\BB(0,R);\RR)$ in $\C_b([0,T-1/R]\times \bar{\BB}(0,R);\RR)$ with respect to this norm.
Then, set 
\[
\widehat{\calD}^{\infty}([0,T) \times \RR^N): = \bigcap_{R \geq 1, n\geq 1}	\widehat{\calD}^n([0,T-1/R] \times \BB(0,R)).
\]
\begin{lemma}
	\label{lem:niceset}
	The function $v(t,x) := P_{t,T}f(x)$ is a member of $ \widehat{\calD}^{\infty}([0,T) \times \RR^N)$ for all $f \in \C_p(\RR^N;\RR)$.
\end{lemma}
\begin{proof}
	We take a sequence $(f_l)_{l \geq 1}$ of smooth approximations of $f$ and associate a $v_l$ to each. For any $n \in \NN$ and any $\alpha \in \A(n)$, we can use the identity \eqref{thebestindentity} to write $\widetilde{W}_{\alpha}(t) v_l(t,x)$ as a linear combination of terms of the form $\widetilde{W}_{[\beta]}(t)v_l(t,x)$ where $\beta \in \A_1(n)$. This allows us to apply the integration by parts formulae in Theorem \ref{th:IBPandGB} to write
	\[
	\widetilde{W}_{\alpha}(t) v_l(t,x) = t^{-\|\alpha\|/2} \, \EE\left[ f_l(X_T^{t,x}) \Phi_{t,\alpha}(T,x) \right]
	\]  
	for some $\Phi_{t,\alpha} \in \K_0(t,\RR)$. This converges over compact subsets of $[0,T) \times \RR^N$.
	%
\end{proof}

The above lemma is used in the next section where we need to perform a stochastic Taylor expansion of $v(t,x):=P_{t,T}f(x)$ for Lipschitz $f$.

\subsection{Stochastic Taylor expansion}
\label{sec:cuberror}

\begin{proposition}
	\label{prop:cuberror}
	Let $f$ be Lipschitz continuous and assume that $\widetilde{UFG}(m)$ holds for some $m \in \mathbb{N}$. Then, $u$, the solution of equation \eqref{eq:PDE} admits a stochastic Taylor expansion for $s<T$
	\begin{align*}
	u(s,X_s^{t,x}) =  \sum_{\alpha \in \A(l)} \widetilde{W}_{\alpha}u (t,x) \, I^{\alpha}_{t,s}(1) + R(l,t,s,x),
	\end{align*}
	with the following estimate on the remainder
	\begin{equation}
		\sup_{x \in \RR^N} \left\| R(l,s,t,x) \right\|_{2} \leq C \sum_{k=l+1}^{l+2} \, (T-s)^{-(k-1)/2}(s-t)^{-k/2}.
	\end{equation}
	This leads to a one-step cubature error estimate of
		\begin{equation}
		\label{eq:onestepcubappendix}
		\sup_{x \in \RR^N} \left| \EE \left[ u(s,X_s^{t,x}) \right] - \EE_{\QQ_{t,s}} \left[ u(s,X_s^{t,x}) \right] \right| \leq C \sum_{k=l+1}^{l+2} \, (T-s)^{-(k-1)/2}(s-t)^{-k/2}.
		\end{equation}
\end{proposition}
\begin{proof}
	For any $g \in \C^{\infty}_b([t,s] \times \RR^N;\RR)$, the following Stratonovich-Taylor expansion is contained in, for example, Kloeden \& Platen \cite[Theorem 5.6.1]{kloedenplaten} 
	\begin{align*}
	g(s,X_s^{t,x}) =  \sum_{\alpha \in \A(l)} \widetilde{W}_{\alpha}g (t,x) \, I^{\alpha}_{t,s}(1) + R(l,t,s,x)
	\end{align*}
	where 
	\begin{equation*}
	R(l,t,s,x,g) = \sum_{-\beta \in \A(l) , \beta \notin \A(l)} I^{\beta}_{t,s} \left(\widetilde{W}_{\beta}g(\cdot,X_{\cdot}^{t,x})\right).
	\end{equation*}
	It is not immediate that this expansion is valid for $g=u$, the solution of equation \eqref{eq:PDE} since it is not differentiable in all directions. However, the Stratonovich-Taylor expansion follows from repeated application of the It\^{o} formula contained in Lemma \ref{lem:Itoformula}. We recall Lemma \ref{lem:niceset}, which says that $(t,x)\mapsto P_{t,T}f(x) \in \widehat{\calD}^{\infty}([0,T) \times \RR^N)$. This guarantees we can apply It\^{o}'s formula as many times as we wish and so the Stratonovich-Taylor expansion is still valid.
	We then have the following estimate for $g=u$
	\begin{align}
	\nonumber	\sup_{x \in \RR^N} \left\| R(l,s,t,x) \right\|_{2} & \leq  \,  \sum_{-\beta \in \A(l) , \beta \notin \A(l)} \left \| I^{\beta}_{t,s}\left(\widetilde{W}_{\beta}u(\cdot,X_{\cdot}^{t,x}) \right) \right\|_{2}  \\ 
	\nonumber	& \leq \sum_{j=l+1}^{l+2} \; \sup_{\beta \in \A(j) } \sup_{(p,x) \in [t,s] \times \RR^N} \left| \widetilde{W}_{\beta}u (p,x) \right| \,  \left\| I^{\beta}_{t,s}(1) \right\|_{2} \\
	\label{eq:Pest}	& \leq C \sum_{j=l+1}^{l+2} \; \sup_{\beta \in \A(j) } \sup_{(p,x) \in [t,s] \times \RR^N} \left| \widetilde{W}_{\beta}u (p,x) \right|  (s-t)^{j/2},
	\end{align}
	where we have used the standard moment estimate on iterated Stratonovich integrals $\left\| I^{\beta}_{t,s}(1) \right\|_{2} \leq C (s-t)^{\| \beta \|/2} $.
	A similar estimate holds under the one step cubature measure, $\QQ_{t,s}$ :
	\begin{equation}
	\label{eq:Qest}
	\sup_{x \in \RR^N} \left| \EE_{\QQ_{t,s}} R(l,t,s,x)  \right| \leq C \sum_{j=l+1}^{l+2} \; \sup_{\beta \in \A(j) } \sup_{(p,x) \in [t,s] \times \RR^N} \left| \widetilde{W}_{\beta}u (p,x) \right| (s-t)^{j/2}.
	\end{equation}
	This is a standard estimate on iterated integrals of bounded variation paths. The constant $C$ depends on $d$, $l$ and the length of the cubature paths. Inequalities \eqref{eq:Pest} and \eqref{eq:Qest} give us control over the error in approximating $P_{t,s}g(t,\cdot)$ by $Q_{t,s}g(t, \cdot)$, 
	\begin{align}
	\nonumber	\sup_{x \in \RR^N} \left| \EE \left[ u(s,X_s^{t,x}) \right] - \EE_{\QQ_{t,s}} \left[ u(s,X_s^{t,x}) \right] \right|  &= \sup_{x \in \RR^N} \left| (\EE - \EE_{\QQ_{t,s}}) R(l,t,s,x) \right| \\
	\label{eq:onesteperror}			& \leq C \, \sum_{j=l+1}^{l+2} \; \sup_{\beta \in \A(j) } \sup_{(p,x) \in [t,s] \times \RR^N} \left| \widetilde{W}_{\beta}u (p,x) \right| (s-t)^{j/2}.
	\end{align}
	 To  bound
	\[
	\sup_{(p,x) \in [t_j,t_{j+1}] \times \RR^N} \left| \widetilde{W}_{\beta}u (p,x) \right|,
	\]
we use the estimate provided in \eqref{eq:PDEbdLip} and taking the supremum over $p \in [t_j, t_{j+1}]$, we get
	\[
	\sup_{(p,x) \in [t_j, t_{j+1}] \times \RR^N}\sup_{\beta \in \A(j) }  \left| \widetilde{W}_{\beta}u (p,x) \right| \leq C \; (T-t_{j+1})^{(1-j)/2}.
	\]
	
\end{proof}


\bibliography{mybiblio}

\end{document}